\documentclass[a4paper,12pt]{article}
\usepackage{mathtools}
\usepackage{amsfonts}
\usepackage{amsthm}
\usepackage{graphicx}
\usepackage{framed}
\usepackage{float}
\usepackage{upgreek}
\usepackage{enumerate}
\usepackage{MnSymbol}
\usepackage{amsbsy}
\usepackage{fullpage}
\usepackage{bbm}
\usepackage[multiple]{footmisc}
\usepackage{tikz}
\usetikzlibrary{matrix}
\usepackage{url}
\theoremstyle{plain}
\newtheorem{thm}{Theorem}[subsection]
\newtheorem{lemma}[thm]{Lemma}
\newtheorem{cor}[thm]{Corollary}
\newtheorem{prop}[thm]{Proposition}
\newtheorem{athm}{Theorem A$\!\!\!$}
\newtheorem{alemma}[athm]{Lemma A$\!\!\!$}
\newtheorem{acor}[athm]{Corollary A$\!\!\!$}
\newtheorem{aprop}[athm]{Proposition A$\!\!\!$}

\newtheorem*{theo}{Theorem}
\newtheorem*{p1}{Proposition 4.1}
\newtheorem*{p2}{Proposition 4.2}
\newtheorem*{l3}{Lemma 4.3}
\newtheorem*{t4}{Theorem 4.4}

\theoremstyle{definition}
\newtheorem{defi}[thm]{Definition}
\newtheorem{rmk}[thm]{Remark}

\newtheorem{adefi}[athm]{Definition A$\!\!\!$}
\newtheorem{armk}[athm]{Remark A$\!\!\!$}

\newtheorem*{r5}{Remark 4.5}
\newtheorem*{r6}{Remark 4.6}

\newtheorem*{defn}{Definition}

\theoremstyle{remark}

\begin{document}

\title{Necessary and Sufficient Conditions for Stable Synchronisation in Random Dynamical Systems}
\author{Julian Newman}
\maketitle

\begin{abstract}
\noindent For a product of i.i.d.~random maps or a memoryless stochastic flow on a compact space $X$, we find conditions under which the presence of locally asymptotically stable trajectories (e.g.~as given by negative Lyapunov exponents) implies almost-sure mutual convergence of any given pair of trajectories (``synchronisation''). Namely, we find that synchronisation occurs and is stable if and only if the system exhibits the following properties: (i)~there is a \emph{smallest} deterministic invariant set $K \subset X$, (ii)~any two points in $K$ are capable of being moved closer together, and (iii)~$K$ admits asymptotically stable trajectories. Our first condition (for which unique ergodicity of the one-point transition probabilities is sufficient) replaces the intricate vector field conditions assumed in Baxendale's similar result of 1991, where (working on a compact manifold) sufficient conditions are given for synchronisation to occur in a SDE with negative Lyapunov exponents.
\end{abstract}

\setcounter{section}{-1}

\section{Introduction and Review}

In this extended Introduction, we will introduce the topic of synchronisation in random dynamical systems, and describe the contribution of our present result in the light of existing knowledge. First, we will present the motivating phenomenon of ``noise-induced synchronisation''; we will then introduce our notions of synchronisation and ``stable\footnote{This refers to stability of synchronisation under perturbations of trajectories, \emph{not} persistence of synchronisation under perturbation of parameters in the underlying model (which is most likely a related, but nonetheless different, interesting problem).} synchronisation'' in the context of random dynamical systems; and with this, we will state our result. We will then review other related notions of stability and synchronisation, and some results concerning these, and will finish with a comparison between our result and a similar result of Baxendale in 1991. (An outline of the structure of our proof will be seen within this comparison.)
\\ \\
Subsequent to the above, the rest of the paper will be structured as follows: In Section~1, we will introduce all the essential notions of random dynamical systems and invariance of (deterministic) sets. We specifically work with ``memoryless noise''; for this, it seems that the most logical framework is that of ``filtered random dynamical systems'' (roughly, as in Section~2.3 of [Arn98]). In Section~2, we introduce the main ideas relating to mutual convergence of trajectories. In Section~3, we formulate and prove our main result. In Section~4, a simple example (based on [LeJan87] and [Kai93]) is studied where our theorem can naturally be applied. In the Appendix, we prove some technical results concerning local stability in non-autonomous and random dynamical systems.

\subsection{Noise-Induced Synchronisation and Our Result}

\subsubsection{The phenomenon of ``noise-induced synchronisation''}

There is a long and well-established history to the study of processes whose evolution over time is governed by time-homogeneous driving forces involving the effects of some memoryless stationary noise.\footnote{We use the term ``noise'' loosely, to refer to any random process in time that contributes to the evolution of a system. A noise process is ``memoryless'' if the statistical distribution for the future behaviour of the noise is completely independent of how the noise has behaved up to the present. The terms ``time-homogeneous'' and ``stationary'' essentially mean ``described by a model that is invariant under time-translation'' (i.e.~a model that is not dependent on any reference time). A standard example of continuous-time memoryless stationary noise is \emph{Gaussian white noise}.} (In essence, this is what the theory of ``homogeneous Markov process'' is concerned with.)
\\ \\
However, in more recent decades, there has been a strong interest in investigating the \emph{simultaneous} time-evolution of the states of \emph{two or more} such processes (each sharing the same range of possible states), where these processes start at different initial states but evolve according to exactly the same laws, simultaneously under exposure to the same sources of noise. (Beyond being subjected to the same noise, let us here assume no further form of coupling between the processes.)
\\ \\
A simple type of example would be an array of identical non-interacting one-dimensional self-oscillators, simultaneously subjected to a sequence of sharp impulses occurring as the jumps of a compound Poisson process. (This example is a particular case of the setup considered in [Pik84], where ``noise-induced synchronisation'' seems to have first been discovered. A multidimensional case is also considered in [NANTK05].)
\\ \\
In the study of such systems, one point of particular interest is the curious phenomenon that, given enough time, the states of the different processes \emph{eventually synchronise with one another}. Admittedly, ``coupling-induced synchronisation'' has been known of for hundreds of years (perhaps going back to Christian Huygens' \emph{Horologium Oscillatorium} of 1673, where it was documented that pendulums suspended from a common beam synchronise); but it is intriguing that this rather ``capricious'' coupling---namely, coupling by exposure to a common unpredictable random influence---should (with positive probability) create this kind of ``order'' in the evolutions of the processes. Now we should say, it is no surprise that the processes should synchronise in a situation where, if we removed all noise, the processes would all have naturally settled towards a common stable equilibrium state anyway. Outside of such situations, the phenomenon that we have described is referred to as \emph{synchronisation by noise} or \emph{noise-induced synchronisation}.
\\ \\
(``Synchronisation by noise'' can also be studied in contexts that do not fit our particular framework described above---e.g.~where the noise has memory, as in [KFI12], or where there is some additional coupling between the processes and/or the processes evolve under different laws, as in [LP13] and [LDLK10]. Partial consideration of such cases was also given in [Pik84]. However, all such cases are outside the scope of this present paper.)

\subsubsection{Random dynamical systems}

Now within the framework that we have described above, one may be able to regard the different processes as \emph{different simultaneous trajectories of one noise-dependent flow} on the space of possible states that the individual processes can attain.
\\ \\
Such a noise-dependent flow is referred to as a ``random dynamical system'' (RDS). More specfically: a random dynamical system is a non-deterministic dynamical system whose non-determinism is due to some time-homogeneous dependence upon the realisation of some stationary noise process. (Mathematically precise formulations of this concept can be found in Section~1.1 of [Arn98], or indeed, Section~1.1 of this paper.) We say that the RDS \emph{has independent increments}, or is \emph{memoryless}, if the noise is itself memoryless (or at least, if the statistics of the RDS are as though the noise were memoryless).
\\ \\
Standard types of random dynamical systems include: (i)~the flow of an autonomous SDE driven by a continuous stochastic process with stationary increments (so the noise process is the ``generalised time-derivative'' of this continuous stochastic process); (ii)~the flow of an autonomous ODE interspersed with ``unpredictable random kicks'' (as in [Chuesh02], Example~1.2.2), or along the same lines, the flow of an autonomous SDE driven by a non-Gaussian L\'{e}vy process (as in Chapter~6 of [App04]); (iii)~in the case of \emph{discrete time}, the evolution of the phase space under a sequence of randomly selected self-maps of the phase space (where the ``noise process'' may be this random sequence of maps, or some other stochastic process that determines the sequence of maps). It is important to note that a trajectory of a RDS is determined by \emph{both} its initial state \emph{and} the realised outcome of the noise process.
\\ \\
Now whenever we talk about ``synchronisation'' of processes, we implicitly assume some kind of \emph{distance function} on the space of states that the processes can attain.  Throughout this Introduction, we shall always assume that the phase space of a random (or deterministic) dynamical system is a Borel-measurable subset of a complete separable metric space, and that the dynamical system is itself jointly continuous in space and c\`{a}dl\`{a}g in time (as in Definition~1.1.8).

\subsubsection{Almost sure synchronisation of trajectories}

Given that we have a metric (say, $d$) on the phase space (say, $X$) of a random dynamical system, it makes sense to talk of ``synchronisation'' or ``mutual convergence'' of paths in $X$ (a ``path in $X$'' meaning an $X$-valued function of time). Specifically, we will say that a collection of paths in $X$ \emph{mutually converges} (or \emph{synchronises}) if, for any two paths $\gamma_1$ and $\gamma_2$ from this collection, $d(\gamma_1(t),\gamma_2(t)) \to 0$ as $t \to \infty$. (And we will say that a collection $\Gamma$ of paths in $X$ is \emph{uniformly mutually convergent} if $\,\sup_{\gamma_1,\gamma_2 \in \Gamma} d(\gamma_1(t),\gamma_2(t)) \, \to \, 0$ as $t \to \infty$.)
\\ \\
With this, we will say that a RDS is ``(\emph{globally}) \emph{synchronising}'' if the following holds: \emph{given any finite set of points in the phase space} (either deterministic, or randomly selected independently of the noise), \emph{with full probability the trajectories starting at these points will mutually converge}.
\\ \\
Of course, to show that a RDS is synchronising, one only actually needs to verify that for \emph{any two deterministic points} in the phase space, the subsequent pair of trajectories will mutually converge with full probability. But the heuristic interpretation of global synchronisation is as follows: \emph{given any finite set of $X$-valued Markov processes evolving simultaneously as trajectories of our RDS (under the same realisation of the noise), we can guarantee that these processes will synchronise.}

\subsubsection{Stable synchronisation}

So far, we have defined what it means for a RDS to be \emph{synchronising}. However, \emph{strictly in and of itself}, this concept may be of little if any practical use. Rather, what we somehow need is a notion of ``stable synchronisation'', where it is ensured that the presence of some tiny unaccounted-for agitations will not prevent synchronous behaviour from being realised.
\\ \\
The simplest illustration of this kind of issue is the distinction between ``attractive fixed points'' and ``asymptotically stable fixed points'' of a deterministic autonomous dynamical system (ADS). For this, the standard example is the discrete-time dynamical system on the circle defined by repeated iteration of an orientation-preserving circle homeomorphism $f$ with a unique fixed point. In the local vicinity of the fixed point, it is easy to show that the dynamics are topologically equivalent to the dynamics of the map
\[ \tilde{f}:\mathbb{R} \to \mathbb{R}, \hspace{4mm} \tilde{f}(x) \, = \, \left\{ \begin{array}{c l}  \tfrac{1}{2}x & x \geq 0 \\ 2x & x \leq 0\end{array} \right. \]
\noindent about its fixed point 0; in other words, the fixed point of $f$ is \emph{attracting on one side and repelling on the other}. So the fixed point is certainly not a ``stable equilibrium'' in any meaningful sense: the future evolution of a trajectory depends sensitively on its initial deviation from the fixed point. (This is formalised by saying that there is a ``local sensitivity'' value $\Delta>0$ such that \emph{every} neighbourhood of the fixed point contains an initial condition whose subsequent trajectory escapes the $\Delta$-neighbourhood of the fixed point. In general, a fixed point of an ADS for which this is not the case is said to be \emph{Lyapunov stable}.)
\\ \\
And yet, since the phase space of our dynamical system just happens to be a \emph{circle} rather than the real line, all those trajectories that are initially repelled away from the fixed point will eventually make their way round the circle to the attracting side of the fixed point. So the fixed point is a ``locally attractive fixed point'', in the na\"{i}ve sense that all trajectories starting near the fixed point eventually converge to the fixed point. (In fact, the fixed point is ``globally attractive'', in that \emph{every} trajectory of the system converges to the fixed point.) However, this statement says \emph{nothing} about \emph{local dynamics near the fixed point} (since, indeed, 0 is not a ``locally attractive fixed point'' of $\tilde{f}$). So then, this na\"{i}ve notion of ``local attractivity'' is, in general, a property of the ``larger-than-local scale'' dynamics! And neither ``local'' nor ``global'' attractivity imply any kind of local \emph{stability} at all.
\\ \\
From a practical point of view, failing to have attractivity within the \emph{local} dynamics about the fixed point can be severely problematic: Imagine a physical process whose time-evolution is described by our dynamical system on the circle. (As we are in discrete time, this will be a process that takes place in discrete steps.) In ``theory'', this process is guaranteed to settle towards the equilibrium state of the system (represented by the fixed point of $f$); and yet in practice, the process will \emph{never} eventually settle around the equilibrium state, since it will always be subject to some kind of unaccounted-for perturbing forces, that will push it over from the attractive side of the equilibrium to the repulsive side. In heuristic terms: the equilibrium state to which the process is attracted cannot be viewed as a ``state of locally minimal energy''.
\\ \\
Now at the root of the above discrepancy between ``local attractivity'' in the na\"{i}ve sense and ``local attractivity'' in a \emph{practically useful} sense is the fact that in the na\"{i}ve version, there is no kind of upper bound on how long it takes for a very nearby trajectory to converge to the fixed point. Indeed in our example, ironically, the \emph{closer} a trajectory is to the fixed point (on the repulsive side), the \emph{longer} it will take to converge to the fixed point!
\\ \\
So then, we can formalise our notion of ``practically useful local attractivity'' as follows: a fixed point of an ADS will be called (\emph{asymptotically}) \emph{stable} if there is a neighbourhood of the fixed point such that the set of trajectories starting in this neighbourhood converges \emph{uniformly} to the fixed point. In other words, we do not allow arbitrarily long convergence times for trajectories starting near the fixed point. In particular, this will guarantee Lyapunov stability.\footnote{Provided the phase space is locally compact, our definition of asymptotic stability is equivalent to the more commonly given definition, that the fixed point is both locally attractive in the na\"{i}ve sense and Lyapunov stable. If the phase space is not locally compact, our definition is at least as strong as the common definition. See the Appendix for details.} A fixed point of an ADS will be called \emph{asymptotically stable in the large} if it is asymptotically stable and every trajectory of the system converges to the fixed point. (Note that a fixed point of an invertible ADS on a compact space can never be asymptotically stable in the large.)
\\ \\
Now our same dynamical system on the circle also demonstrates the problem with na\"{i}vely considering ``synchronisation'' while ignoring the ``stability'' of the synchronisation: suppose we have \emph{two} simultaneous physical processes, both of whose time-evolutions are described by this same dynamical system. Again, in ``theory'' the processes will synchronise, since they will both be attracted to the equilibrium state; however, in practice, due to the ``unaccounted-for perturbing forces'' that we have mentioned, the two processes will always become desynchronised by the time they have reached sufficiently close to the equilibrium.
\\ \\
This motivates the need to define a notion of ``stable synchronisation'' for random dynamical systems. Now the very fact that the future evolution of the system is unpredictable, and (at least in the case of memoryless noise) \emph{will never become any more predictable as time progresses}, places a limitation on how strong a definition of ``stable'' is feasible to work with. Practically speaking, if we required ``almost sure upper bounds'' on the mutual convergence times of nearby trajectories, we would end up ruling out all cases where noise is responsible for creating synchronisation (which is precisely the phenomenon that has motivated our whole study).
\\ \\
Nonetheless, we can still find quite powerful notions of ``stable synchronisation'' that are satisfied surprisingly often in practice. Very crudely speaking, we will want to say that a synchronising RDS is \emph{stably} synchronising if, as the largest possible amount of unaccounted-for perturbation that can occur tends to 0, the likelihood that such perturbation severely impedes synchronisation tends to 0. How best to formalise this is not necessarily straightforward, and may depend on context; but for a \emph{most basic} definition, it will be reasonable to work with the following:

\begin{defn}
A memoryless RDS on a separable complete metric space will be called (\emph{globally}) \emph{stably synchronising} if it is both synchronising and ``everywhere stable'' in the sense that for all $\varepsilon>0$ there exists $\delta>0$ such that for any given subset of the phase space of diameter less than $\delta$, with a probability of at least $1-\varepsilon$ the set of trajectories starting in this subset will be uniformly mutually convergent.
\end{defn}

\noindent (In the case that the phase space is \emph{compact}, there are equivalent definitions to the above given in Sections~0.2.1 and 0.2.2.)
\\ \\
In this definition there is still an element of the same kind of danger that we described further above, namely that two given initial conditions can be arbitrarily close and yet can (with positive probability) take any length of time to synchronise---and can even separate any distance apart, before synchronising! As we have already said, ruling out all such possibility is infeasible; in fact, for invertible RDS on a compact space, it is literally impossible!\footnote{It is clear that a random product or stochastic flow of homeomorphisms on a compact space \emph{cannot} be both globally synchronising and have the property that for all $x \in X$ and $\varepsilon>0$ there exists $\delta>0$ such that for any $y \in X$ with $d(x,y) < \delta$, with \emph{full} probability the subsequent trajectories of $x$ and $y$ remain within a distance of $\varepsilon$ from each other.}
\\ \\
Finally, it is worth noting that if a RDS is synchronising, it does \emph{not} necessarily follow that with full probability \emph{all} trajectories of the RDS mutually converge. Rather, it is only \emph{after} we have selected some finite (or countable) set of initial conditions that with full probability the subsequent trajectories mutually converge. That said, if a memoryless RDS is stably synchronising then we \emph{can} conclude that with full probability there is a \emph{dense open set} of initial conditions whose subsequent trajectories mutually converge. Nonetheless there will still exist, in many cases, a random nowhere-dense unstable set. (Indeed, this \emph{must} be the case for a stably synchronising invertible RDS on a compact space.)

\subsubsection{Negative Lyapunov exponents vs. synchronisation}

In this paper, we will focus specifically on memoryless RDS on a \emph{compact} space. Since we only address the question of ``global'' synchronisation (where \emph{every} pair of initial conditions has to be taken into account), there are important cases that will not be covered by our results and may require some further investigation. For example, suppose we have a product of i.i.d.~random order-preserving homeomorphisms on a closed interval; obviously the endpoints of the interval remain fixed, and so one can only consider the question of synchronisation within the \emph{interior} of the interval---which is not compact. For such a situation, our results in this paper do not say anything (although for that particular case some results are already known, as we shall soon mention).
\\ \\
Now since we assume a compact phase space, the family of Markov transition probabilities describing individual trajectories of the RDS (henceforth called the ``one-point transition probabilities'') admits at least one ergodic probability measure $\rho$. Assuming some differentiable structure on the phase space, and provided the RDS is itself sufficiently regular and well-behaved with respect to this differentiable structure, one can talk of Lyapunov exponents, and the following is known to be true (see Section~0.2.4): if the list of Lyapunov exponents associated to $\rho$ consists of only \emph{negative} exponents, then with full probability the trajectory of $\rho$-almost every initial condition in the phase space will be \emph{exponentially stable}. (Given a realisation of the noise, a trajectory of the RDS is called ``exponentially stable'' if there is a neighbourhood of its initial position such that the set of trajectories starting in this neighbourhood is uniformly mutually convergent at an exponential rate; when we have uniform mutual convergence at a rate that is not necessarily exponential, we simply say that the trajectory is \emph{asymptotically stable}.\footnote{Unlike the case of a fixed point of an autonomous dynamical system, our definition of asymptotic stability is no longer precisely equivalent to the usual definition, even on a locally compact phase space---rather, it is very slightly stronger. See the Appendix for details. (However, our definition of \emph{exponential} stability is the standard definition.)})
\\ \\
It is commonly said that ``negativity of the Lyapunov exponents implies synchronisation''. However, it is really only a \emph{local synchronisation} property that is implied by the negativity of the Lyapunov exponents associated to $\rho$---namely, that (as a simple corollary of the above) for $\rho$-almost every initial condition $x$, given any sequence of initial conditions $(y_n)_{n \in \mathbb{N}}$ tending to $x$, the probability that the trajectories of $x$ and $y_n$ mutually converge tends to 1 as $n \to \infty$.
\\ \\
The natural question, then, is to find conditions under which we can pass from such \emph{local} properties to global stable synchronisation.

\subsubsection{Our main result}

In this paper, we provide an answer to the above question, namely:

\begin{theo}
Given a memoryless RDS on a compact metric space $X$, and a stationary probability measure $\rho$ for its one-point transition probabilities, the RDS is stably synchronising if and only if the following all hold:
\begin{enumerate}[\indent (i)]
\item every (deterministic) non-empty closed subset of $X$ that is almost-surely forward-invariant under the RDS has $\rho$-full measure;$\,$\footnote{This is equivalent to saying that the support of $\rho$ is the only non-empty closed subset of $X$ on which $\varphi$ has ``minimal dynamics'' (in the sense of Definition~1.2.5).}
\item any two points in the support of $\rho$ have a positive probability of being moved closer together (in forward time) under the RDS;
\item with positive probability there is at least one asymptotically stable trajectory within the support of $\rho$.
\end{enumerate}
\noindent Moreover, if the RDS is synchronising, then $\rho$ is the only stationary probability measure.
\end{theo}

\noindent In the context of differentiable RDS and ``Lyapunov exponents'', condition (iii) is implied by the negativity of all the Lyapunov exponents associated to $\rho$ (assuming that $\rho$ is ergodic; if $\rho$ is not ergodic, then one can consider the ergodic decomposition of $\rho$---see Remark~2.2.12). It is also worth saying that in general, $\rho$ being the only stationary probability measure is a \emph{stronger} statement than condition (i). (Hence, \emph{given} (iii), conditions (i) and (ii) serve as necessary and sufficient conditions for the RDS to be synchronising.)
\\ \\
The above theorem (excluding the final assertion) is an equivalent formulation of Theorem~3.2.1 of this paper (the difference being that Theorem~3.2.1 is expressed without explicit reference to a measure on the phase space). Condition~(i) can actually be expressed as a condition on the one-point transition probabilities (since, by a continuity argument, forward-invariance with respect to the one-point transition probabilities is equivalent to almost-sure forward-invariance under the RDS). Condition~(ii) can be expressed as a condition on the \emph{two}-point transition probabilities. (This ``two-point contractibility'' described in condition~(ii) has been considered by Baxendale and Stroock---see Proposition~4.1 of [BS88], or condition~(4.1) in [Bax91].)
\\ \\
So in essence, our theorem says that under an appropriate one-point condition and an appropriate two-point condition, one can pass from local-scale synchronisation to global-scale synchronisation. For a brief discussion of the essential points of the proof, see Section~0.3.
\\ \\
(The final assertion in the above theorem---which is Proposition~2.1.4 of this paper---is a straightforward result; implicitly, it was proved in Corollary~2 of [KN04], although the proof that we shall give is much more elementary and direct.)

\subsection{A Review of Synchronous Behaviour in Random Systems}

Crudely speaking, the way in which noise induces synchronous behaviour in a dynamical system is by altering the proportions of time that trajectories spend in contractive regions of the phase space versus the time spent in expansive regions. It turns out that even \emph{chaotic} dynamical systems can be made to exhibit large-scale synchronous behaviour when noise is added;\footnote{The general concept that the addition of noise can induce some kind of order out of chaos was perhaps first reported in [MT83].} some examples of such are described in [TMHP01].
\\ \\
We will introduce some of the different types of synchronous behaviour that can be studied in RDS. (There does not yet seem to exist standard nomenclature for some of these concepts, and so for convenience, much of the terminology used will be our own.) In all that follows, a ``stationary (resp.~ergodic) probability measure'' for a memoryless RDS refers to a stationary (resp.~ergodic) probability measure on $X$ for the family of (one-point) transition probabilities defined by the RDS.

\subsubsection{Concepts of stability and synchronisation}

In the analytic study of RDS, questions of stability and synchronisation will often divide into two categories: (1) Does the RDS exhibit stable behaviour (in some sense) within the local vicinity of some/most/all trajectories? (2) Assuming we know enough about the RDS's local behaviour near individual trajectories, can we somehow ``extend'' this to deduce that on a ``larger-than-local'' scale, different trajectories will approach each other in the long run? (This paper is concerned with the second category.)
\\ \\
Examples of ``local'' behaviour about trajectories include the notions of Lyapunov, asymptotic and exponential stability of trajectories, and the notion of ``Lyapunov exponents'' (which will be discussed more in Section~0.2.4). For convenience, we will say that a memoryless RDS is ``$\rho$-almost everywhere asymptotically [resp.~exponentially] stable'' (where $\rho$ is a stationary probability measure) if it is almost surely the case that for $\rho$-almost every initial condition, the subsequent trajectory will be asymptotically [resp.~exponentially] stable.
\\ \\
For a memoryless RDS, the ``larger-than-local scale'' concepts that we shall discuss are:
\begin{enumerate}[\indent (a)]
\item global synchronisation and global stable synchronisation, as defined earlier (but we will generally omit the word ``global'');
\item ``$\rho$-almost-everywhere stable synchronisation''---given a stationary probability measure $\rho$, this is the phenomenon that with full probability there is an open set of $\rho$-full measure such that all trajectories starting in this set are asymptotically stable and synchronise with each other;
\item equivalence classes defined by the following random (i.e.~noise-dependent) equivalence relation on the phase space: two initial conditions are \emph{equivalent} if their subsequent trajectories mutually converge;
\item ``statistical synchronisation'' with respect to a stationary probability measure $\rho$---this is the phenomenon that the distance between the trajectories of two randomly selected initial conditions, each selected with distribution $\rho$ independently of each other and of the noise, converges in probability to 0 as time tends to infinity.
\end{enumerate}

\noindent The use of terminology set out in (a)--(d) above is all our own; however, as we shall see, most of these concepts have been studied before (sometimes under different but equivalent definitions).
\\ \\
On a compact phase space, a memoryless RDS is globally stably synchronising if and only if it is both globally synchronising and $\rho$-almost everywhere asymptotically stable (where $\rho$ is the unique stationary probability measure).\footnote{This can be justified by Theorem~2.2.14 of this paper.} More generally, given a stationary probability measure $\rho$, we have the following implications:

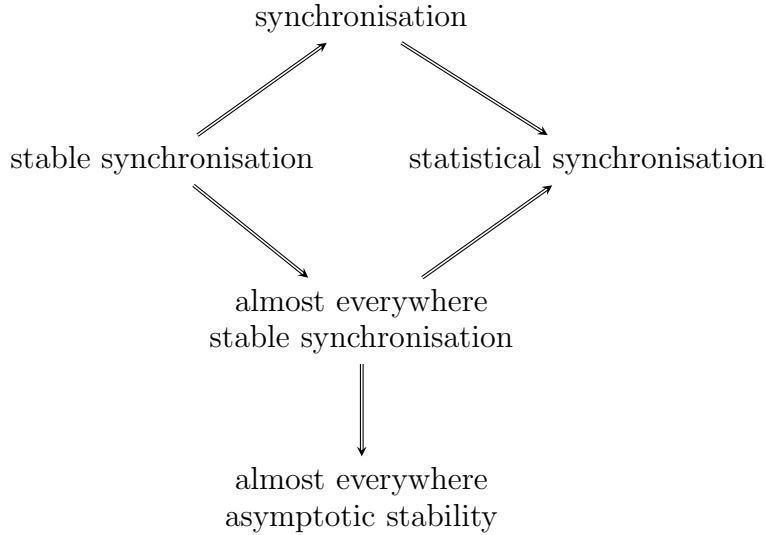
\begin{center}
\hspace{-7mm}
\begin{tikzpicture}
  \matrix (m) [matrix of math nodes,row sep=3em,column sep=-4em,minimum width=2em] {
      & \textrm{synchronisation} &  \\
   \textrm{stable synchronisation} & & \textrm{statistical synchronisation} \\
      & \shortstack{\textrm{almost everywhere} \\ \textrm{stable synchronisation}} & \\
      & \shortstack{\textrm{almost everywhere} \\ \textrm{asymptotic stability}} & \\};
  \path[-stealth]
    (m-2-1) edge [double] (m-1-2)
    (m-2-1) edge [double] (m-3-2)
    (m-1-2) edge [double] (m-2-3)
    (m-3-2) edge [double] (m-2-3)
    (m-3-2) edge [double] (m-4-2);
\end{tikzpicture}
\end{center}

\noindent For a compact phase space, all the concepts in the above diagram do not rely on a given \emph{metric} on the phase space, but depend only on the \emph{topology} of the phase space; in other words, they are \emph{preserved} under switching between different \emph{topologically equivalent} metrics on the phase space. (Similarly, for smooth RDS on a compact smooth manifold, Lyapunov exponents are preserved under switching between different Riemannian metrics.)
\\ \\
However, when working on a non-compact space, the concepts in the above diagram---except statistical synchronisation---are \emph{not} topological concepts, but depend on the ``uniform structure'' on the phase space (that is, they are not always preserved under switchting between topologically equivalent metrics, but they are always preserved under switching between uniformly equivalent metrics\footnote{Two metrics $d$ and $d'$ on a set $X$ are \emph{uniformly equivalent} if the identity function on $X$ is both uniformly $(d,d')$-continuous and uniformly $(d',d)$-continuous.}). For example: on a non-compact smooth manifold, given \emph{any} pair of paths that escapes every compact set, one can always ``stretch out'' the Riemannian metric on the manifold to a sufficient extent that the paths do not mutually converge. (This kind of argument does not apply to \emph{statistical} synchronisation, because statistical synchronisation is based on convergence in probability: even if a trajectory escapes every compact set, at any one given time it may only have a very small probability of being outside a sufficiently large compact set.)
\\ \\
Statistical synchronisation is quite remarkable, in that it is not only preserved by any switch between \emph{topologically equivalent} metrics on the phase space, but is even preserved by any switch between \emph{measurably equivalent} separable metrics on the phase space! (The author has not actually seen this result explicitly elsewhere, but it follows essentially immediately from the proofs of known results. See Section~0.2.2 for more detail.)

\subsubsection{Relation to random fixed points}

\noindent In Section~0.1.4, we hinted at a kind of parallel between the notion of synchronisation of trajectories (in random or deterministic systems) and attracting fixed points of deterministic systems. It turns out that there is quite a strong mathematical basis behind this parallel, which we shall describe now.
\\ \\
For a given RDS, a ``random trajectory'' will mean an assignment, to each possible realisation of the noise, of one trajectory of the RDS under that same realisation of the noise.
\\ \\
We consider a random dynamical system, where the underlying noise process is modelled as having been going on since eternity past (so each ``noise realisation'' consists of both a \emph{future} and a \emph{past}). In this case, there is a special kind of random trajectory, that serves as the RDS-analogue of the concept of a fixed point of an ADS; the initial condition of such a random trajectory (i.e.~its location at time 0---which is a random variable) is called a \emph{random fixed point} or an \emph{equilibrium}. Precise definitions are given in Remark~2.3.2.
\\ \\
Now if the location of a random fixed point is determined (modulo zero-probability events) just by the \emph{past} of the noise, then we will say that it is ``past-measurable''. Using the definition given in Remark~2.3.2, it is easy to show that for a memoryless RDS, the law of any past-measurable random fixed point is a stationary probability measure. In the converse direction ([KS12],~Theorem~4.2.9), given any stationary probability measure $\rho$ there exists a unique (modulo zero-probability events) past-measurable ``random invariant measure'' whose expectation is $\rho$---again, see Remark~2.3.2 for precise definitions---and if this ``random invariant measure'' is a random Dirac mass, then the random variable on which it is concentrated is a past-measurable random fixed point.
\\ \\
Given a stationary probability measure $\rho$ for a memoryless RDS, we have the following results linking synchronisation and random fixed points:

\begin{itemize}
\item The RDS is statistically synchronising with respect to $\rho$ if and only if there exists a past-measurable random fixed point whose law is $\rho$ (due to Proposition~2.6(i) of [Bax91]).
\item The RDS is $\rho$-almost everywhere stably synchronising if and only if there exists a past-measurable random fixed point whose law is $\rho$ and whose (forward-time) trajectory is asymptotically stable almost surely (by Proposition~3 of [LeJan87] with $n=1$).
\item If the phase space is compact, then the following are equivalent:\footnote{The equivalence of (i) and (ii) is due to Remark~2.2.7 of this paper; the equivalence of (i) and (iii) is due to Remark~2.2.16 of this paper.}
\begin{enumerate}[\indent (i)]
\item the RDS is stably synchronising;
\item the RDS is synchronising, and for every initial condition in the phase space, the subsequent trajectory is almost surely asymptotically stable;
\item there is a past-measurable random fixed point, whose trajectory is asymptotically stable almost surely, with the additional property that every initial condition in the phase space has full probability of belonging to the basin of (forward-time) attraction of the trajectory of this random fixed point.
\end{enumerate}
\end{itemize}

\noindent So in a sense: ``almost sure stable synchronisation'' with respect to a stationary probability measure $\rho$ (in the context of \emph{R}DS) is the \emph{stochastic equivalent} of ``asymptotic stability'' for a fixed point $p$ (in the context of \emph{A}DS); and similarly, ``global stable synchronisation'' can be seen as the stochastic equivalent of the existence of a fixed point that is asymptotically stable in the large. (Although asymptotic stability in the large is impossible for invertible ADS on a compact space, global stable synchronisation is certainly possible for invertible RDS on a compact space.)
\\ \\
Now it is worth mentioning that the definition of a random fixed point does not make any reference to topological concepts, but only to measurable concepts. Moreover, the existence and uniqueness of a past-measurable random invariant measure with expectation $\rho$ does not require any continuity properties of the RDS.\footnote{The construction of the random invariant measure, as given in Theorem~4.2.9 of [KS12], does assume a Polish topology on the phase space, in order for the limit involved to be meaningful (and the ``exceptional null set'' in the construction depends on which Polish topology is taken). However, continuity of the RDS with respect to this topology is not needed for the construction to work, \emph{nor} is it actually needed for the constructed measure to be invariant: to show that the measure $\mu$ with disintegration $(\mu_\omega)$ is invariant under $\Theta^t$, it suffices to show that for all $\tau>0$ the measures $\mu|_{\mathcal{F}_{\!\!-(\tau+t)}^\infty \otimes \mathcal{B}(X)}$ and $\Theta^t_\ast\!\left(\mu|_{\mathcal{F}_{\!-\tau}^\infty \otimes \mathcal{B}(X)}\right)$ on $\mathcal{F}_{\!\!-(\tau+t)}^\infty \otimes \mathcal{B}(X)$ are equal, which is straightforward by considering disintegrations with respect to the restricted probability space $(\Omega,\mathcal{F}_{\!\!-(\tau+t)}^\infty,\mathbb{P}|_{\mathcal{F}_{\!\!-(\tau+t)}^\infty})$.} Consequently, statistical synchronisation is really a property of \emph{measurable RDS on standard measurable spaces}.

\subsubsection{Special cases for verifying synchronisation}

We mentioned in Section~0.2.1 that one may ask first about the local stability properties of a RDS and then, from there, address the issue of larger-scale synchronous behaviour. However, before we look at local stability, let us first mention some cases where one can address the issue of larger-scale synchronous behaviour more directly:
\\ \\
(I)~If one can show that there is a finite time $T$ such that for all $0<t<T$ the time-$t$ map of the RDS is almost surely a contraction, then it is clear that the RDS will be globally synchronising, and this synchronisation will be stable in any reasonable sense. (Such a situation is impossible for invertible RDS on a compact space, since a contraction on a compact space cannot be surjective.)
\\ \\
(II)~In the context of \emph{invertible} RDS on a \emph{circle}, one may be able to find conditions guaranteeing some large-scale synchronous behaviour without making any direct reference to local dynamics. For example, Theorem~1 of [KN04] considers a product of i.i.d.~random orientation-preserving circle homeomorphisms, and gives sufficient conditions for synchronisation (indeed, the conditions given are sufficient for stable synchronisation). In the case of a memoryless orientation-preserving diffeomorphic RDS on the circle with an ergodic probability measure $\rho$, provided the system is ``reasonably well behaved'' (see e.g.~(H1) and (H2)\footnote{Beware that in [LeJan87] the characters $||^2$ are missing from the end of the denominator in the formula for $\delta_2(T)$.} in [LeJan87]), we have the following: if $\rho$ is equivalent to the Lebesgue measure then, unless $\rho$ is almost-surely preserved under the flow, the Lyapunov exponent associated to $\rho$ is automatically guaranteed to be negative (e.g.~by Proposition~1(b) of [LeJan87]); and so all the general results about deducing large-scale synchronisation from negative Lyapunov exponents immediately apply. (This is exemplified in the comments after Remark~4.5 of this paper.)
\\ \\
(III)~For a memoryless monotone RDS on $\mathbb{R}$ or a Borel subset thereof, if there is an ergodic probability measure $\rho$ then the RDS is guaranteed to be statistically synchronising with respect to $\rho$ (by Theorem~18.4(iv) of [Arn98], combined with the characterisation of statistical synchronisation given in Section~0.2.2). Similar results can also be found for monotone RDS on higher-dimensional spaces (see e.g.~Theorem~1 of [CS04], and for an application of this result see Proposition~5.6 of [CCK07]). Also, just as for the circle in case (II), given a memoryless diffeomorphic monotone RDS on a compact interval $[a,b]$ with an ergodic probability measure $\rho$, provided the system is ``reasonably well behaved'' we have the following: if $\rho$ is equivalent to the Lebesgue measure then the Lyapunov exponent associated to $\rho$ is automatically guaranteed to be negative; consequently, we can immediately conclude that the RDS is synchronising on $(a,b)$ (see Sections~0.2.4 and 0.2.5).
\\ \\
It is worth saying that in \emph{continuous} time, for a memoryless diffeomorphic RDS on a line/circle with continuous trajectories, any non-Dirac ergodic probability measure $\rho$ must be equivalent to the Lebesgue measure restricted to $\mathrm{supp}\,\rho$. (Hence, as exemplified in Corollary~4.4 of [Crau02], the assertions in (II) and (III) can be made without explicit statement of the condition of equivalence to the Lebesuge measure.)

\subsubsection{Local-scale synchronous behaviour (``stability'')}

Suppose we have a memoryless smooth RDS (that is: spatially smooth, with partial derivatives that are jointly continuous in space and time) on a smooth manifold $X$, and let $\rho$ be an ergodic probability measure. Provided the RDS has sufficiently well-controlled first-order spatial derivatives, there will exist a finite list of \emph{Lyapunov exponents} (heuristically speaking, ``exponential rates of repulsion on the infinitesimal scale'') that is common to the trajectories of $(\rho \otimes \mathbb{P})$-almost all combinations of an initial condition and a noise realisation (see [Arn98], Theorem~4.2.6).
\\ \\
In such a context, when addressing the question of local stability the first consideration will generally be to establish the \emph{sign of the maximal Lyapunov exponent}. For a smooth RDS on a compact manifold with sufficiently well-controlled (higher-order) spatial derivatives, it is known (see e.g.~the start of Section~3 of [LeJan87]) that if all the Lyapunov exponents are negative then the system is $\rho$-almost everywhere exponentially stable; in particular, this implies that with full probability there will exist a partition of $\rho$-almost the whole of $X$ into some open regions such that all trajectories starting in the same region mutually converge. This result is essentially the ``codimension-0 case of the random stable manifold theorem''. Similar statements can also be obtained for RDS on non-compact manifolds (see e.g.~Corollary~3.1.1 and Remark~(iii) of [MS99], which considers SDEs on Euclidean space)\footnote{For a general extensive treatment of random invariant manifolds, see Chapter~7 of [Arn98].}.
\\ \\
(It is worth saying that, although Lyapunov exponents are defined as an ``asymptotic property'', strict negativity of the maximal Lyapunov exponent can easily be expressed as a ``finite-time property''---see e.g.~Remark~2.2.12 of this paper.)

\subsubsection{Larger-scale synchronous behaviour}

For all the following results, we assume that we have a memoryless RDS with an ergodic probability measure $\rho$.
\\ \\
\textbf{(A)} Let us first mention (in the context of smooth RDS) the case of a \emph{null maximal Lyapunov exponent}. The task of establishing whether there is stable or synchronous behaviour of any kind when the maximal Lyapunov exponent associated to $\rho$ is exactly zero can be notoriously difficult; and (as far as the author is aware) there is no reason to expect that this case can be ``dismissed as being the complement of the generic situation'' (rather, for evidence to the contrary, see [BBD14]).
\\ \\
One surprising result where a conclusion is obtained from a null maximal Lyapunov exponent is Theorem~5.8 (and in particular its consequence, Corollary~5.12) of [Bax91]. Here, the subject under consideration is the flow of an autonomous SDE on a compact connected manifold $X$ (of dimension greater than 1) driven by a multidimensional Wiener process. A non-degeneracy condition is required, namely that the H\"{o}rmander bracket condition is satisfied by the set of diffusion coefficients in the ``lifted'' SDE (on the manifold of non-zero tangent vectors on $X$) describing the time-evolution of a non-trivial tangent vector under the total spatial derivative of the flow of the original SDE (see equation~(2.4) of [BS88]). This condition guarantees in particular that $\rho$ is the only stationary probability measure and is equivalent to the Riemannian volume measure.
\\ \\
Corollary~5.12 of [Bax91] essentially states that under this condition, if the maximal Lyapunov exponent is zero and any two points $X$ have a positive probability of being moved closer together under the flow (``two-point contractibility''), then the system is statistically synchronising with respect to $\rho$. (The two-point contractibility condition is formulated in a different but equivalent manner---see Section~3.1 for further details.)
\\ \\
\textbf{(B)} Perhaps the first major result linking local asymptotic stability to larger-scale synchronous behaviour was provided by Propositions~2~and~3 of [LeJan87] (which work in discrete time and on a compact manifold, but also generalise to continuous time and to non-compact spaces). Here it was shown that if the system is $\rho$-almost everywhere asymptotically stable (as implied by negativity of the maximal Lyapunov exponent in Section~0.2.4), then the partition of mutual convergence described in Section~0.2.4 will (in its coarsest form) almost surely consist of $n$ open regions of equal $\rho$-measure $\frac{1}{n}$, where $n$ is independent of the realisation of the noise. The case that $n=1$ is precisely the case that the RDS is $\rho$-almost everywhere stably synchronising.
\\ \\
\textbf{(C)} The ``$n$'' in Le~Jan's result corresponds to the number of random atoms in the past-measurable random invariant measure whose expectation is $\rho$ (which will be purely atomic if the system is $\rho$-almost everywhere asymptotically stable); for a \emph{monotone} RDS on $\mathbb{R}$, this random invariant measure will always be a random Dirac mass. (This is an equivalent formulation of the first assertion of (III) in Section~0.2.3). Consequently, given a monotone RDS on $\mathbb{R}$ that is $\rho$-almost everywhere asymptotically stable, the ``$n$'' in Le~Jan's result is equal to 1. By monotonicity, this actually implies that the RDS is synchronising on the interior of $\mathrm{supp}\,\rho$; in particular, if $\rho$ has full support then the RDS is globally stably synchronising.
\\ \\
Using this fact, one can find situations where the addition of noise to a non-synchronising monotone dynamical system induces synchronisation. This is well exemplified in [CF98], where it is shown that the presence of noise can ``destroy a pitchfork bifurcation''; namely, adding white noise to the right-hand side of the ODE $\,dX_t = (\alpha X_t - X_t^3)dt\,$ has the effect of ``blurring'' the two asymptotically stable fixed points for $\alpha>0$ into one random fixed point that is asymptotically stable in the large, just as is present for $\alpha \leq 0$. Moreover in [CDLR13], ``random topological conjugacies'' (as in Definition~1.9.8 of [Arn98]) were shown to exist between any pair of SDEs in the $\alpha$-indexed family of SDEs (viewed as one stochastic family of differential equations simultaneously driven by the same noise process) defined by varying $\alpha$ in a sufficiently small neighbourhood of 0. So after the addition of noise, there is a kind of ``stochastic structural stability'' about $\alpha=0$.
\\ \\
Nonetheless, as pointed out in [CDLR13], one cannot claim that all qualitiative changes in behaviour across the critical parameter-value $\alpha=0$ disappear after the addition of noise: for $\alpha<0$, after the addition of noise there \emph{continues to exist} an exponentially decaying almost-sure upper bound on the distance between the trajectories of any two given initial conditions; meanwhile for $\alpha>0$, after the addition of noise the system \emph{continues not to have this property}. In fact, for $\alpha>0$ the trajectories of \emph{any} two distinct points $x$ and $y$ could take an arbitrarily long time to synchronise. (This last point was not explicitly proved in [CDLR13], but is easy to see: since the ergodic measure has full support and the system has continuous trajectories, there will come a time at which 0 lies between the trajectories of $x$ and $y$---and then the system can behave like the deterministic system for any length of time.) To represent this qualitative change across $\alpha=0$, a new type of local structural stability for RDS was introduced in [CDLR13], which was shown to fail for the family of SDEs under consideration: namely, it was shown that any ``random topological conjugacy'' between a negative $\alpha$-value and a positive $\alpha$-value cannot be bicontinuous at the random fixed point \emph{uniformly} across a full-measure set of noise realisations. (This loss of ``equicontinuous structural stability'' was linked to the change in sign of the maximum of the dichotomy spectrum.)
\\ \\
\textbf{(D)} The result that is most directly related to our main result in this paper is Theorem~4.10(i) of [Bax91]. Once again, a SDE on a compact connected manifold $X$ driven by a $d$-dimensional Wiener process (for some $d \in \mathbb{N}$) is considered. Again, a non-degeneracy condition is required, namely that H\"{o}rmander's condition and a kind of ``minimal dynamics'' condition\footnote{specifically: given any initial condition and any non-empty open set, there is a $C^1$ sample path for the driving Wiener process under which the forward-trajectory of the initial condition will reach the open set.} hold for the ``lifted'' SDE (on the unit sphere bundle of $X$) describing the normalised evolution of a non-zero tangent vector under the total spatial derivative of the flow of the original SDE (see equation~(2.6) of [BS88]). As stated in [BS88] and [Bax91], this condition is weaker than the non-degeneracy condition assumed for Theorem~5.8 of [Bax91], but it still guarantees that $\rho$ is only stationary probability measure and is equivalent to the Riemannian volume measure.
\\ \\
Theorem~4.10(i) of [Bax91] states that under this condition, if the Lyapunov exponents are all negative and any two points in $X$ have a positive probability of being moved closer together under the flow, then the system is synchronising. (Indeed, the system is stably synchronising, since synchronisation and $\rho$-almost everywhere asymptotic stability together imply stable synchronisation.)

\subsection{Comparison of Baxendale's result with our result}

For a memoryless RDS, we will use the term ``invariant set'' to mean a non-empty closed deterministic set that is almost surely forward-invariant under the RDS. Our main result (as given in Section~0.1.6) improves upon the above result of Baxendale ([Bax91], Theorem~4.10) in a few important ways.
\\ \\
Firstly, and most importantly, we have replaced the SDE-specific non-degeneracy condition with the (simpler and weaker) condition that $\mathrm{supp}\,\rho$ is the smallest invariant set. This is important because it allows the result to be no longer specific to SDEs, but also applicable to discrete-time and even non-invertible random dynamical systems. Our proof makes no reference to notions from stochastic calculus.
\\ \\
A second important improvement is that our conditions are necessary and sufficient for stable synchronisation (and moreover, \emph{given} the knowledge of $\rho$-almost everywhere asymptotic stability, our conditions are necessary and sufficient for synchronisation). The conditions in [Bax91] are not necessary conditions for synchronisation, since there exist stably synchronising SDEs whose stationary probability measure does not have full support. (Take, for example, the equation $\,dX_t = \sin(2 \pi X_t)dt + \cos(2 \pi X_t) \!\! \circ \!\! dW_t\,$ on the unit circle: the invariant sets are the whole circle and the interval $[\frac{1}{4},\frac{3}{4}]$, and so there is an ergodic probability measure whose support is $[\frac{1}{4},\frac{3}{4}]$; using the content of Section~0.2.3, the system must be synchronising on $(\frac{1}{4},\frac{3}{4})$, and then it is easy to use our main result to show that the system is stably synchronising on the whole circle.)
\\ \\
Furthermore, we do not require one to verify the two-point contractibility condition on the whole phase space $X$ but only on $\mathrm{supp}\,\rho$, which we do \emph{not} require to be equal to the whole of $X$. (Now it is worth mentioning: the assumption of two-point contractibility on the whole of $X$ automatically implies that there is a smallest invariant set; but the support of a given ergodic probability measure might not be equal to the smallest invariant set, unless we happen to know that there is only one ergodic probability measure.)
\\ \\
A (relatively minor) additional improvement is the fact that we do not require any form of \emph{exponential} stability (while the proof in [Bax91] genuinely makes use of the \emph{strict} negativity of all Lyapunov exponents, beyond the mere fact that this implies asymptotic stability).
\\ \\
The last improvement to mention is that our proof is considerably technically simpler and much more elementary than in [Bax91]. The \emph{``key fact''} underlying our proof is the (rather intuitive) fact that given any initial condition for a memoryless RDS, if with positive probability the subsequent trajectory stays forever inside some compact set $C$, then $C$ must contain an invariant set (see e.g.~Lemma~1.2.7 of this paper).
\\ \\
Now the proof in [Bax91] also indirectly uses this fact (applied to the \emph{two-point motion}, which is an RDS on $X \times X$) via citation of Proposition~4.1 of [BS88], where (using a very similar argument to our proof of Lemma~1.2.7) it is proved that under the two-point contractibility condition on $X$, the subsequent trajectories of any given pair of initial conditions will reach arbitrarily small distances of each other with full probability.\footnote{Proposition~4.1 of [BS88] actually gives a statistical statement about how long this will take, but that is not relevant here.} This result is then combined with a preliminary result (Theorem~4.6 of [Bax91], which essentially states that the system is everywhere locally synchronising) to yield global synchronisation. (All the real complexities of Baxendale's proof are within the proof of this preliminary result, although as we shall soon mention, even this preliminary result can be proved more straightforwardly.)
\\ \\
However, our above ``key fact'' actually yields a \emph{stronger} conclusion than the result that [Bax91] cites from Proposition~4.1 of [BS88]. Namely, it yields the following: the two-point contractibility condition on $X$ implies that the subsequent trajectories of any given pair of initial conditions will, with full probability, simultaneously reach an arbitrarily small distance of \emph{any} given point within the smallest invariant set. It immediately follows that \emph{if} there is a point within the smallest invariant set whose subsequent trajectory is almost surely asymptotically stable (as implied by $\rho$ being the unique stationary probability measure and having negative Lyapunov exponents), then the system is synchronising.
\\ \\
Thus, on the one hand, we can straightforwardly prove Baxendale's synchronisation result, without the preliminary Theorem~4.6. But on the other hand, the \emph{stability} of synchronisation is precisely a strengthened form of the same Theorem~4.6. Nonetheless, this actually also turns out to be fairly straightforward to prove (see Theorem~2.2.14 of this paper) using our same ``key fact'' (applied now to the \emph{original} RDS, not the two-point motion) together with almost-everywhere asymptotic stability.
\\ \\
In the above, we have essentially laid out how \emph{our} result will be proved; only, slightly more work than this will be needed, due to our conditions being weaker than in [Bax91].

\section{Memoryless Random Dynamical Systems}

Let $\mathbb{T}^+$ denote either $\mathbb{N} \cup \{0\}$ or $[0,\infty)$. Let $(X,d)$ be a compact metric space. Given a filtration $(\mathcal{F}_t)_{t \in \mathbb{T}^+}$ of $\sigma$-algebras on a set $\Omega$, we write $\mathcal{F}_\infty:=\sigma(\mathcal{F}_t:t \in \mathbb{T}^+)$.

\subsection{Basic definitions}

\begin{defi}
A \emph{dynamical system} $(\theta^t)_{t \in \mathbb{T}^+}$ on a measurable space $(\Omega,\mathcal{F})$ is a $\mathbb{T}^+$-indexed family of measurable maps $\theta^t:\Omega \to \Omega$ such that $\theta^0$ is the identity function and $\theta^{s+t}=\theta^s \circ \theta^t$ for all $s,t, \in \mathbb{T}^+$.
\end{defi}

\noindent Given a dynamical system $(\theta^t)$ and a set $E \subset \Omega$, for any $t \in \mathbb{T}^+$ we write $\theta^{-t}(E)$ to denote the preimage of $E$ under $\theta^t$.

\begin{defi}
A \emph{shift dynamical system} (or \emph{filtered dynamical system}) on a filtered measurable space $(\Omega,\mathcal{F},(\mathcal{F}_t)_{t \in \mathbb{T}^+})$ is a dynamical system $(\theta^t)$ on the measurable space $(\Omega,\mathcal{F})$ with the property that for all $s,t \in \mathbb{T}^+$, $\theta^t$ serves as a measurable function from $(\Omega,\mathcal{F}_{s+t})$ to $(\Omega,\mathcal{F}_s)$.
\end{defi}

\noindent Note that if $(\theta^t)$ is a shift dynamical system on $(\Omega,\mathcal{F},(\mathcal{F}_t))$, then $(\theta^t)$ can also be regarded as a dynamical system on the measurable space $(\Omega,\mathcal{F}_\infty)$.

\begin{defi}
Given a dynamical system $(\theta^t)$ on a measurable space $(\Omega,\mathcal{F})$, a probability measure $\mathbb{P}$ on $\Omega$ is said to be \emph{invariant under $(\theta^t)$} if $\theta^t_\ast\mathbb{P}=\mathbb{P}$ for all $t \in \mathbb{T}^+$.
\end{defi}

\begin{defi}
Given a dynamical system $(\theta^t)$ on a measurable space $(\Omega,\mathcal{F})$, a set $E \in \mathcal{F}$ is said to be \emph{forward-invariant} (resp. \emph{backward-invariant}) \emph{under $(\theta^t)$} if $\theta^t(E) \subset E$ (resp. $\theta^{-t}(E) \subset E$) for all $t \in \mathbb{T}^+$. Given a $(\theta^t)$-invariant probability measure $\mathbb{P}$, a set $E \in \mathcal{F}$ is said to be \emph{$\mathbb{P}$-almost invariant under $(\theta^t)$} if $\,\mathbb{P}(E \setminus \theta^{-t}(E))=0$, or equivalently $\mathbb{P}(\theta^{-t}(E) \setminus E)=0$, for all $t \in \mathbb{T}^+$.
\end{defi}

\noindent Note that a set $E \in \mathcal{F}$ is forward-invariant (resp. $\mathbb{P}$-almost invariant) if and only if its complement $\Omega \setminus E$ is backward-invariant (resp. $\mathbb{P}$-almost invariant).

\begin{defi}
Given a dynamical system $(\theta^t)$ on a measurable space $(\Omega,\mathcal{F})$, a probability measure $\mathbb{P}$ on $\Omega$ is said to be \emph{ergodic with respect to $(\theta^t)$} if $\mathbb{P}$ is invariant and for every $\mathbb{P}$-almost invariant set $E \in \mathcal{F}$ under $(\theta^t)$, either $\mathbb{P}(E)=0$ or $\mathbb{P}(E)=1$. Equivalently (as in Proposition~7.2.4 of [FM10]), $\mathbb{P}$ is ergodic if and only if it is an extremal point of the convex set of invariant probability measures.
\end{defi}

\noindent Note that for a measurable map $T:\Omega \to \Omega$, a probability measure $\mathbb{P}$ is invariant under the discrete-time dynamical system $(T^n)_{n=0}^\infty$ if and only if $T_\ast\mathbb{P}=\mathbb{P}$. We will say that $\mathbb{P}$ is invariant under (resp.~ergodic with respect to) $T$ if it is invariant under (resp.~ergodic with respect to) $(T^n)_{n=0}^\infty$.

\begin{defi}
A (\emph{stationary}) \emph{noise space} $(\Omega,\mathcal{F},(\mathcal{F}_t),\mathbb{P},(\theta^t))$ consists of a shift dynamical system $(\theta^t)$ on a filtered measurable space $(\Omega,\mathcal{F},(\mathcal{F}_t))$, together with a $(\theta^t)$-invariant probability measure $\mathbb{P}$. We will say that the noise space $(\Omega,\mathcal{F},(\mathcal{F}_t),\mathbb{P},(\theta^t))$ is \emph{memoryless} if for all $s \in \mathbb{T}^+$ the $\sigma$-algebras $\mathcal{F}_s$ and $\theta^{-s}\mathcal{F}_\infty$ are independent under $\mathbb{P}$.
\end{defi}

\begin{rmk}
It is easy to show that for a memoryless noise space $(\Omega,\mathcal{F},(\mathcal{F}_t),\mathbb{P},(\theta^t))$, the restricted probability measure $\mathbb{P}|_{\mathcal{F}_\infty}$ is ergodic with respect to $\theta^t$ for all $t \in \mathbb{T}^+ \setminus \{0\}$ (and hence, in particular, is ergodic with respect to $(\theta^t)_{t \in \mathbb{T}^+}$). See Corollary~133 of [New15] for a proof.
\end{rmk}

\begin{defi}
Given a noise space $(\Omega,\mathcal{F},(\mathcal{F}_t),\mathbb{P},(\theta^t))$, a (\emph{filtered, c\`{a}dl\`{a}g}) \emph{random dynamical system $\,\varphi = \left(\varphi(t,\omega)\right)_{t \in \mathbb{T}^+ \! , \, \omega \in \Omega}\,$ on $X$} is a $(\mathbb{T}^+ \! \times \Omega)$-indexed family of continuous functions $\varphi(t,\omega):X \to X$ such that
\begin{enumerate}[\indent (a)]
\item the map $(\omega,x) \mapsto \varphi(t,\omega)x$ is $(\mathcal{F}_t \otimes \mathcal{B}(X))$-measurable for each $t \in \mathbb{T}^+$;
\item $\varphi(0,\omega)$ is the identity function for all $\omega \in \Omega$;
\item $\varphi(s+t,\omega) \, = \, \varphi(t,\theta^s\omega) \circ \varphi(s,\omega)\,$ for all $s,t \in \mathbb{T}^+$ and $\omega \in \Omega$;
\item for any decreasing sequence $(t_n)$ in $\mathbb{T}^+$ converging to a value $t$, and any sequence $(x_n)$ in $X$ converging to a point $x$, $\,\varphi(t_n,\omega)x_n \to \varphi(t,\omega)x\,$ as $n \to \infty$ for all $\omega \in \Omega$;
\item there exists a function $\varphi_-:\mathbb{T}^+ \times \Omega \times X \to X$ such that for any strictly increasing sequence $(t_n)$ in $\mathbb{T}^+$ converging to a value $t$, and any sequence $(x_n)$ in $X$ converging to a point $x$, $\,\varphi(t_n,\omega)x_n \to \varphi_-(t,\omega,x)\,$ as $n \to \infty$ for all $\omega \in \Omega$.
\end{enumerate}
\end{defi}

\noindent Condition (a) is the condition of being $(\mathcal{F}_t)$-filtered; conditions (b) and (c) constitute the so-called ``cocyle property'', and represent temporal consistency (like the ``flow equations'' for a deterministic dynamical system); conditions (d) and (e) make up the ``c\`{a}dl\`{a}g'' property, with (d) being right-continuity and (e) being left limits. (As it happens, the only point in this paper where the ``left limits'' property is directly used is in the Appendix: it is needed in order to guarantee that our definition of ``asymptotic stability'' implies Lyapunov stability.) Note that the map $(t,\omega,x) \mapsto \varphi(t,\omega,x)$ is jointly measurable (e.g.~by Lemma~16(B) of [New15]).
\\ \\
A ``memoryless RDS'' simply means a RDS defined over a memoryless noise space.
\\ \\
In the introduction, we assumed that it makes sense to regard the underlying noise process as having no beginning in time; mathematically, a two-sided-time noise process is represented by a noise space $(\Omega,\mathcal{F},(\mathcal{F}_t),\mathbb{P},(\theta^t))$ where for every $t \in \mathbb{T}^+$, $\theta^t:\Omega \to \Omega$ is a measurable isomorphism of $(\Omega,\mathcal{F})$.
\\ \\
\textbf{From now on, we will always work with a RDS $\varphi$ on $X$ defined over a memoryless noise space $(\Omega,\mathcal{F},(\mathcal{F}_t),\mathbb{P},(\theta^t))$.}

\begin{defi}
We define the \emph{two-point motion} $\varphi \! \times \! \varphi$ of $\varphi$ to be the $(\mathbb{T}^+ \! \times \Omega)$-indexed family of functions $\varphi \! \times \! \varphi(t,\omega):X \times X \to X \times X$ given by
\[ \varphi \! \times \! \varphi(t,\omega)(x,y) \ = \ (\varphi(t,\omega)x,\varphi(t,\omega)y). \]
\end{defi}

\noindent It is easy to show that $\varphi \! \times \! \varphi$ is a RDS on $X \times X$ over the noise space $(\Omega,\mathcal{F},(\mathcal{F}_t),\mathbb{P},(\theta^t))$.

\subsection{Invariance of sets}

For each $x \in X$ and $t \in \mathbb{T}^+$, define the probability measure $\varphi_x^t$ on $X$ by
\[ \varphi_x^t(A) \ = \ \mathbb{P}(\omega : \varphi(t,\omega)x \in A) \ = \ \mathbb{P}(\omega : \varphi(t,\theta^s\omega)x \in A) \hspace{4mm} \textrm{(for any $s$)}. \]
\noindent It is not hard to show that $(\varphi_x^t)_{x \in X \! , \, t \in \mathbb{T}^+}$ defines a Markovian family of transition probabilities on $X$, and that for all $y \in X$ the stochastic process $\left( \varphi(t,\cdot)y \right)_{t \in \mathbb{T}^+}$ is an $(\mathcal{F}_t)$-adapted Markov process with these transition probabilities. It is also easy to show, using the dominated convergence theorem, that the map $(t,x) \mapsto \varphi_x^t$ is jointly continuous in $x$ and right-continuous in $t$ (where the set of probability measures on $X$ is equipped with the narrow topology).

\begin{defi}
Given a point $x \in X$ and an open set $U \subset X$, we will say that $U$ is \emph{accessible from $x$} (\emph{under $\varphi$}) if
\[ \mathbb{P}( \, \omega \, : \, \exists \, t \in \mathbb{T}^+ \textrm{ s.t. } \varphi(t,\omega)x \in U) \ > \ 0.\]
\end{defi}

\noindent By considering rational times (and using the right-continuity of trajectories of $\varphi$), it is clear that $U$ is accessible from $x$ if and only if there exists $t \in \mathbb{T}^+$ such that $\varphi_x^t(U)>0$. Now given a point $x \in X$, we have (by the second-countability of $X$) that an arbitrary union of open sets that are not accessible from $x$ is itself not accessible from $x$. So let $U_x$ denote the largest open set that is not accessible from $x$.

\begin{defi}
We will say that a closed set $K \subset X$ is \emph{forward-invariant} (\emph{under $\varphi$}) if
\[ \mathbb{P}( \, \omega \, : \, \forall \, t \in \mathbb{T}^+, \, \varphi(t,\omega)K \subset K) \ = \ 1 \, ; \]
\noindent and we will say that an open set $U \subset X$ is \emph{backward-invariant} if $X \setminus U$ is forward-invariant.
\end{defi}

\noindent By considering rational times and a countable dense subset of $K$ (and using the continuity properties of $\varphi$), it is easy to see that $K$ is forward-invariant if and only if $\varphi_x^t(K)=1$ for all $x \in K$ and $t \in \mathbb{T}^+$. In particular, it follows that $K$ is forward-invariant if and only if for all $x \in K$, $X \setminus K$ is not accessible from $x$.
\\ \\
By the second-countability of $X$, an arbitrary intersection of closed forward-invariant sets is itself forward-invariant. (In fact, the set of open backward-invariant sets forms a topology on $X$.) So, for any $x \in X$, let $G_x$ be the smallest closed forward-invariant set containing $x$. It is obvious that $X \setminus G_x$ is \emph{contained} in $U_x\hspace{0.2mm}$; but due to continuity, we actually have \emph{equality}:

\begin{lemma}
For any $x \in X$, $X \setminus G_x \ = \ U_x$.
\end{lemma}

\begin{proof}
We need to show that $U_x \, \subset \, X \setminus G_x$, for which it is sufficient to show that $U_x$ is backward-invariant. So fix any $y \in X \setminus U_x$, and suppose for a contradiction that there exists $t \in \mathbb{T}^+$ such that $\varphi_y^t(U_x)>0$. Since $U_x$ is open, the map $\xi \mapsto \varphi_\xi^t(U_x)$ is lower semicontinuous, and so there exists a neighbourhood $V$ of $y$ such that $\varphi_\xi^t(U_x)>0$ for all $\xi \in V$. Since $y \nin U_x$, $V$ is accessible from $x$, so there exists $s \in \mathbb{T}^+$ such that $P_x^s(V)>0$. Hence
\[ \varphi_x^{s+t}(U_x) \ = \ \int_X \varphi_\xi^t(U_x) \, \varphi_x^s(d\xi) \ \geq \ \int_V \varphi_\xi^t(U_x) \, \varphi_x^s(d\xi) \ > \ 0 \, , \]
\noindent contradicting the fact that $U_x$ is not accessible from $x$.
\end{proof}

\noindent Just as $G_x \subset X$ denotes the smallest forward-invariant closed set under $\varphi$ containing $x$, so we will also write $G_{(x,y)} \subset X \times X$ to denote the smallest forward-invariant closed set under $\varphi \! \times \! \varphi$ containing $(x,y)$.

\begin{lemma}
For any $x,y \in X$, the image of $G_{(x,y)}$ under the projection $(u,v) \mapsto u$ (resp.~$(u,v) \mapsto v$) is $G_x$ (resp.~$G_y$).
\end{lemma}

\begin{proof}
Let $A_{x,y} \subset X$ be the image of $G{(x,y)}$ under $(u,v) \mapsto u$. $A_{x,y}$ is a closed set, since $G_{(x,y)}$ is compact. Also $A_{x,y}$ is forward-invariant: for any $u \in A_{x,y}$, if we let $v$ be such that $(u,v) \in G_{(x,y)}$, then $(X \setminus A_{x,y}) \times X$ is not accessible from $(u,v)$, and so $X \setminus A_{x,y}$ is not accessible from $u$. Finally, if $B$ is a closed proper subset of $A_{x,y}$ with $x \in B$, then $B$ is not forward-invariant: $(X \setminus B) \times X$ is accessible from $(x,y)$ by Lemma~1.2.3 (applied to $\varphi \! \times \! \varphi$), and so $X \setminus B$ is accessible from $x$. Thus we have shown that $A_{x,y}$ is the smallest closed forward-invariant set containing $x$, as required.
\end{proof}

\begin{defi}
We will say that a set $K \subset X$ is \emph{minimal} (\emph{under $\varphi$}), or that $\varphi$ \emph{has minimal dynamics on $K$}, if $K$ is a non-empty closed forward-invariant set and the only closed forward-invariant subsets of $K$ are $K$ and $\emptyset$.
\end{defi}

\noindent Note that for any non-empty closed $K \subset X$ the following are equivalent:
\begin{itemize}
\item $K$ is minimal;
\item $K$ is a minimal element (with respect to set-inclusion) of the set of non-empty closed forward-invariant subsets of $X$;
\item for all $x \in K$, $G_x=K$;
\item $K$ is forward-invariant, and for any $x \in K$ and any open set $U \subset X$ with $U \cap K \neq \emptyset$, $U$ is accessible from $x$.
\end{itemize}

\noindent The following is broadly based on \textrm{[}KH95\textrm, solution to Exercise~3.3.4 (p768){]}:

\begin{prop}
Every non-empty closed forward-invariant set contains a minimal set.
\end{prop}

\begin{proof}
Fix a non-empty closed forward-invariant set $C \subset X$. For any non-empty closed forward-invariant $M \subset C$, let
\[ m(M) \ := \ \sup\{d_H(M,G) \, : \, G \neq \emptyset, \, G \subset M \textrm{ closed forward-invariant}\} \]
\noindent where $d_H$ denotes the Hausdorff (semi-)distance. Let $C \supset M_1 \supset M_2 \supset M_3 \supset \ldots$ be a nested sequence of non-empty closed forward-invariant sets such that $d_H(M_n,M_{n+1}) \geq \frac{n}{n+1}m(M_n)$ for all $n \in \mathbb{N}$. Cantor's intersection theorem gives that $\,K:=\bigcap_{n=1}^\infty M_n\,$ is non-empty. Now since $X$ is totally bounded, we must have that $d_H(M_n,M_{n+1}) \to 0$ as $n \to \infty$, and so $m(M_n) \to 0$ as $n \to \infty$. It is clear that $m(\cdot)$ is monotone, so it follows that $m(K)=0$. Hence $K$ is minimal.
\end{proof}

\noindent (Incidentally, in the above proof, one can use the continuity of $\varphi$ to show that the supremum in the formula for $m(M)$ is actually a maximum, and hence the prefactors $\frac{n}{n+1}$ can be removed; however, it is much quicker just to do as we have done.)
\\ \\
It follows from Proposition~1.2.6 that if there is a unique minimal set $K \subset X$, then every non-empty closed forward-invariant set contains $K$.

\begin{lemma}
Suppose $K \subset X$ is a closed set possessing no non-empty closed forward-invariant subsets. Then for any $x \in X$, for $\mathbb{P}$-almost every $\omega \in \Omega$ there is an unbounded sequence $(t_n)_{n \in \mathbb{N}}$ in $\mathbb{T}^+$ such that for all $n \in \mathbb{N}$, $\varphi(t_n,\omega)x \nin K$.
\end{lemma}

\noindent The proof of Lemma~1.2.7 is essentially the same as the proof of Proposition~4.1 of [BS88]. We will use the following fact, which is sufficiently clear that we do not write out a proof, but is nonetheless worth stating explicitly:

\begin{lemma}
Let $(M_t)_{t \in \mathbb{T}^+}$ be an $(\mathcal{F}_t)$-adapted $X$-valued homogeneous Markov process with transition probabilities $(P_x^t)_{x \in X \! , \, t \in \mathbb{T}^+}$. Fix $t \in \mathbb{T}^+$, let $D$ be a countable subset of $\mathbb{T}^+$, and let $T:\Omega \to D$ be an $\mathcal{F}_t$-measurable function. Then
\[ \mathbb{E}[\mathbbm{1}_A(M_{t+T})|\mathcal{F}_t] \ = \ P_{M_t}^T(A) \]
\noindent for all $A \in \mathcal{B}(X)$.
\end{lemma}

\begin{proof}[Proof of Lemma 1.2.7]
Let $D:=\mathbb{Q} \cap \mathbb{T}^+$. Fix $x \in X$ and let $M_t(\omega):=\varphi(t,\omega)x$ for all $t$ and $\omega$. For each $y \in K$, $G_y \setminus K$ is non-empty (otherwise $G_y$ would be a non-empty closed forward-invariant subset of $K$) and so there exists $t \in \mathbb{T}^+$ such that $\varphi_y^t(K) < 1$.\footnote{Since the map $y \mapsto \varphi_y^t(K)$ is upper semicontinuous and $K$ is compact, one can show that $t$ can be taken from a \emph{bounded} interval $[0,T]$ (where $T$ is independent of $y$). Consequently (as in Proposition~4.1 of [BS88]) one can obtain a stronger result than stated in Lemma~1.2.7; however, we will not need this.} Now the map $t \mapsto \varphi_y^t(K)$ is right upper semicontinuous for each $y \in X$; and so it follows that for each $y \in X$ there exists $t \in D$ such that $\varphi_y^t(K) < 1$. So, if we define a function $l:K \to [0,1]$ by
\[ l(y) \ := \ \inf_{t \in D} \varphi_y^t(K) \]
\noindent then $l$ is strictly less than 1 on the whole of $K$. Also note that $l$ is upper semicontinuous. Therefore $l$ has a maximum value $c'$, which is strictly less than 1. So, fixing an arbitrary value $c \in (c',1)$, we have that for all $y \in K$ there exists $t \in D$ such that $\varphi_y^t(K) \leq c\,$; in fact, it is easy to see that one can construct a \emph{measurable} function $\tau:K \to D$ such that $\varphi_y^{\tau(y)}(K) \leq c$ for all $y \in K$.\footnote{e.g.~if $(s_n)_{n \in \mathbb{N}}$ is an enumeration of $D$, set $\tau(y):=s_{N(y)}$ where $N(y):=\min\{n \in \mathbb{N}:\varphi_y^{s_n}(K)<1\}$.} We extend $\tau$ to the whole of $X$ by setting $\tau(y)=0$ for all $y \in X \! \setminus \! K$.
\\ \\
Now to obtain the desired result, it is sufficient just to show that for each $N \in \mathbb{N}$, for $\mathbb{P}$-almost every $\omega \in \Omega$ there exists $t \geq N$ such that $\varphi(t,\omega)x \nin K$. Fix any $N \in \mathbb{N}$, and define the sequence $\,(T_n)_{n \in \mathbb{N} \cup \{0\}}$ of functions $T_n:\Omega \to D$ by
\begin{align*}
T_0(\omega) \ &= \ N \\
T_n(\omega) \ &= \ T_{n-1}(\omega) + \tau(M_{T_{n-1}(\omega)}(\omega)) \hspace{4mm} (n \geq 1)
\end{align*}
\noindent for all $\omega \in \Omega$. For each $n \in \mathbb{N} \cup \{0\}$, let $E_n:=\{ \, \omega \in \Omega \, : \, M_{T_r(\omega)}(\omega) \in K \ \textrm{for all} \ 0 \leq r \leq n \, \}$. We will show by induction that for each $n \in \mathbb{N} \cup \{0\}$, $\,\mathbb{P}( E_n ) \, \leq \, c^n$. (Obviously, once we have shown this, we are done.) The $n=0$ case is trivial. Now fix $m \in \mathbb{N}$ such that $\,\mathbb{P}( E_{m-1} ) \, \leq \, c^{m-1}$. It is not hard to show that $E_{m-1} \cap T_{m-1}^{-1}(\{t\}) \in \mathcal{F}_t$ for all $t \in D$; so then,
\begin{align*}
\mathbb{P}(E_m) \ &= \ \int_{E_{m-1}} \mathbbm{1}_K(M_{T_m(\omega)}(\omega)) \, \mathbb{P}(d\omega) \\
&= \ \sum_{t \in D} \, \int_{E_{m-1} \cap T_{m-1}^{-1}(\{t\})} \!\! \mathbbm{1}_K(M_{t+\tau(M_t(\omega))}(\omega)) \, \mathbb{P}(d\omega) \\
&= \ \sum_{t \in D} \, \int_{E_{m-1} \cap T_{m-1}^{-1}(\{t\})} \varphi_{M_t(\omega)}^{\tau(M_t(\omega))}(K) \, \mathbb{P}(d\omega) \hspace{4mm} \textrm{(by Lemma 1.2.8)} \\
&\leq \ \sum_{t \in D} \, \int_{E_{m-1} \cap T_{m-1}^{-1}(\{t\})} c \ \mathbb{P}(d\omega) \\
&= \ c \, \mathbb{P}(E_{m-1}) \\
&\leq \ c^m.
\end{align*}
\noindent So we are done.
\end{proof}

\noindent The following corollary will essentially be the key ingredient in the proof of our main result.

\begin{cor}
Let $K$ be a minimal set, and let $U$ be an open set with $U \cap K \neq \emptyset$. Then for each $x \in K$, for $\mathbb{P}$-almost every $\omega \in \Omega$ there exist (arbitrarily large) times $t \in \mathbb{T}^+$ such that $\varphi(t,\omega)x \in U$. Moreover, if $K$ is the only minimal set, then for each $x \in X$, for $\mathbb{P}$-almost every $\omega \in \Omega$ there exist (arbitrarily large) times $t \in \mathbb{T}^+$ such that $\varphi(t,\omega)x \in U$.
\end{cor}

\begin{proof}
It is clear that $K \setminus U$ is a closed set possessing no non-empty closed forward-invariant subsets; hence Lemma~1.2.7 combined with the forward-invariance of $K$ gives the first statement. If $K$ is the only minimal set, then it is clear that $X \setminus U$ is a closed set possessing no non-empty closed forward-invariant subsets, so Lemma~1.2.7 gives the second statement.
\end{proof}

\noindent Note that the times in Lemma~1.2.7 and Corollary~1.2.9 can be selected to belong to $\mathbb{Q}$. (This follows from the fact that the trajectories of $\varphi$ are right-continuous in time, but can also be seen directly within the proof of Lemma~1.2.7.)

\subsection{Stationary and ergodic measures}

\noindent For any $t \in \mathbb{T}^+$ we define the map $\Theta^t:\Omega \times X \to \Omega \times X$ by $\Theta^t(\omega,x)=(\theta^t\omega,\varphi(t,\omega)x)$. It is easy to show that $\Theta^{s+t}=\Theta^t \circ \Theta^s$ for all $s,t \in \mathbb{T}^+$; in fact, $(\Theta^t)$ serves as a dynamical system on $(\Omega \times X, \mathcal{F}_\infty \otimes \mathcal{B}(X))$.

\begin{defi}
We will say that a probability measure $\rho$ on $X$ is \emph{stationary} if $\mathbb{P}|_{\mathcal{F}_\infty\!} \otimes \rho$ is invariant under $(\Theta^t)$. We will say that $\rho$ is \emph{ergodic} if $\mathbb{P}|_{\mathcal{F}_\infty\!} \otimes \rho$ is ergodic with respect to $(\Theta^t)$. We will say that $\varphi$ is \emph{uniquely ergodic} if there is exactly one stationary probability measure.
\end{defi}

\noindent Stationarity can also be defined with respect to the family of probability measures $(\varphi_x^t)$ defined in Section~1.2: 

\begin{prop}
For any probability measure $\rho$ on $X$ and any $t \in \mathbb{T}^+$, let $\varphi^{t\ast\!}\rho$ be the probability measure on $X$ defined by
\[ \varphi^{t\ast\!}\rho(A) \ = \ \int_\Omega \varphi(t,\omega)_\ast\rho(A) \, \mathbb{P}(d\omega) \ = \ \int_X \varphi_x^t(A) \, \rho(dx) \]
\noindent for all $A \in \mathcal{B}(X)$. Then $\Theta^t_\ast(\mathbb{P}|_{\mathcal{F}_\infty\!} \otimes \rho)=\mathbb{P}|_{\mathcal{F}_\infty\!} \otimes \varphi^{t\ast\!}\rho$ for all $t \in \mathbb{T}^+$. In particular, $\rho$ is stationary if and only if $\varphi^{t\ast\!}\rho=\rho$ for all $t \in \mathbb{T}^+$.
\end{prop}

\noindent For a proof, see Theorem~143(i) of [New15] or (in discrete time) Lemma~1.2.3 of [Kif86].
\\ \\
Likewise, ergodicity can also be defined with respect to the family of probability measures $(\varphi_x^t)$; see Theorem~143(ii) of [New15] or (in discrete time) Theorem~1.2.1 of [Kif86] for details.
\\ \\
We also have the following, which is essentially due to the ``ergodic decomposition theorem'':

\begin{lemma}
For any bounded measurable functions $h_1,h_2:X \to [0,\infty)$, if there exists a stationary probability measure $\rho$ such that $\int_X h_1(x) \, \rho(dx)=0$ and $\int_X h_2(x) \, \rho(dx)>0$, then there exists an ergodic probability measure $\rho'$ such that $\int_X h_1(x) \, \rho'(dx)=0$ and $\int_X h_2(x) \, \rho'(dx)>0$.
\end{lemma}

\begin{proof}
This follows immediately from combining Corollary~109 and Theorem~143 of [New15].
\end{proof}

\begin{prop}
For any stationary probability measure $\rho$, $\mathrm{supp}\,\rho$ is forward-invariant. Conversely, for any non-empty closed forward-invariant set $K \subset X$, there exists a stationary probability measure $\rho$ such that $\rho(K)=1$. (Hence in particular, $\varphi$ must admit at least one stationary probability measure.)
\end{prop}

\begin{proof}
Let $\rho$ be a stationary probability measure. Fix any $t \in \mathbb{T}^+$. By Proposition~1.3.2,
\[ \int_X \varphi_x^t(\mathrm{supp}\,\rho) \, \rho(dx) \ = \ \rho(\mathrm{supp}\,\rho) \ = \ 1, \]
and so
\[ \rho(x \in X : \varphi_x^t(\mathrm{supp}\,\rho)=1) \ = \ 1.\]
\noindent Since the map $x \mapsto \varphi_x^t$ is continuous, the set $\{x \in X : \varphi_x^t(\mathrm{supp}\,\rho)=1\}$ is closed, and hence contains $\mathrm{supp}\,\rho$. Consequently, as in Section~1.2, $\mathrm{supp}\,\rho$ is forward-invariant.
\\ \\
The converse statement is a version of the ``Krylov-Bogolyubov theorem'' (see e.g.~Theorem~114 of [New15] or Lemma~5.2.1 of [Kif86]).
\end{proof}

\begin{cor}
For any minimal set $K \subset X$, there exists an ergodic probability measure $\rho$ such that $\mathrm{supp}\,\rho=K$.
\end{cor}

\begin{proof}
By Proposition~1.3.4, there exists a stationary probability measure $\tilde{\rho}$ such that $\tilde{\rho}(K)=1$. Hence, by Lemma~1.3.3, there exists an ergodic probability measure $\rho$ such that $\rho(K)=1$, i.e.~$\mathrm{supp}\,\rho \subset K$. Since $K$ is minimal, it follows from Proposition~1.3.4 that $\mathrm{supp}\,\rho=K$.
\end{proof}

\section{Synchronisation and Stability}

In this section, we will say that $\varphi$ is an \emph{open mapping RDS} if $\varphi(t,\omega)$ is an open mapping for all $t$ and $\omega$. For example, any invertible RDS is an open mapping RDS. (The results that we shall prove for open mapping RDS are not actually integral to the rest of the paper, but are included for the sake of completeness, since most RDS considered in practice are invertible.)

\subsection{Synchronisation}

\begin{defi}
Given a sample point $\omega \in \Omega$ and a non-empty set $A \subset X$, we will say that \emph{$A$ contracts under $\omega$} (or that \emph{$A$ is uniformly mutually convergent under $\omega$}) if $\,\mathrm{diam}(\varphi(t,\omega)A) \to 0$ as $t \to \infty$. (It is not hard to show that, since $X$ is compact, this does not depend on the metrisation $d$ of the topology on $X$.)
\end{defi}

\begin{defi}
Given $\omega \in \Omega$ and points $x,y \in X$, we will say that \emph{$x$ and $y$ mutually converge under $\omega$} if $\{x,y\}$ contracts under $\omega$.
\end{defi}

\begin{defi}
We will say that $\varphi$ is (\emph{globally}) \emph{synchronising} if for every $x,y \in X$,
\[ \mathbb{P}( \, \omega \, : \, \textrm{$x$ and $y$ mutually converge under $\omega$} \, ) \ = \ 1. \]
\end{defi}

\noindent The following is a generalisation of Corollary~2 of [KN04].

\begin{prop}
If $\varphi$ is synchronising then $\varphi$ is uniquely ergodic.
\end{prop}

\begin{proof}
Fix a point $x \in X$, and let $\rho$ be any stationary probability measure; we will show that $\varphi_x^t$ must converge (in the narrow topology) to $\rho$ as $t \to \infty$, from which it follows that $\rho$ is the only stationary probability measure. Let $g:X \to \mathbb{R}$ be any continuous function. It is clear that for $(\mathbb{P} \otimes \rho)$-almost all $(\omega,y) \in \Omega \times X$, $x$ and $y$ mutually converge under $\omega$, and so $g(\varphi(t,\omega)y) - g(\varphi(t,\omega)x) \to 0$ as $t \to \infty$. The dominated convergence theorem then gives that
\[ \underbrace{\int_{\Omega \times X} g(\varphi(t,\omega)y) \, (\mathbb{P} \otimes \rho)(d(\omega,y))}_{\textcircled{a}} \; - \underbrace{\int_\Omega g(\varphi(t,\omega)x) \, \mathbb{P}(d\omega)}_{\textcircled{b}} \ \; \to \ \; 0 \hspace{3mm} \textrm{as } \, t \to \infty. \]
\noindent Observe, however, that
\[ \textcircled{a} \ = \ \int_X g(z) \, \rho(dz) \]
\noindent since $\mathbb{P}|_{\mathcal{F}_\infty\!} \otimes \rho$ is $(\Theta^t)$-invariant, and that
\[ \hspace{2mm} \textcircled{b} \ = \int_X g(z) \, \varphi_x^t(dz). \]
\noindent So we are done.
\end{proof}

\begin{rmk}
If the phase space $X$ were not compact, then $\varphi$ being synchronising would imply that there is \emph{at most} one stationary probability measure. (In the above proof, we would take $g$ to be a bounded uniformly continuous function.)
\end{rmk}

\subsection{Stability and potential stability}

Note that, just as a trajectory of a deterministic dynamical system is determined by its initial condition, so likewise a trajectory of a random dynamical system is determined by the combination of a noise realisation and an initial condition.

\begin{defi}
We will say that a pair $(\omega,x) \in \Omega \times X$ is \emph{asymptotically stable} if there exists a neighbourhood $U$ of $x$ which contracts under $\omega$. Let $O \subset \Omega \times X$ denote the set of all asymptotically stable pairs $(\omega,x)$. For each $x \in X$, let $O_x := \{ \omega : (\omega,x) \in O \}$ denote the $x$-section of $O$.
\end{defi}

\begin{defi}
For any non-empty open set $U \subset X$, let $E_U \subset \Omega$ be the set of sample points under which $U$ contracts.
\end{defi}

\noindent Note that for any $x \in X$,
\[ O_x \ = \ \bigcup_{n=1}^\infty E_{B_{\frac{1}{n}}(x)} \]
\noindent where $B_r(x)$ denotes the open ball of radius $r$ about $x$.

\begin{lemma}
For any open set $U \subset X$, $E_U$ is $\mathcal{F}_\infty$-measurable. Furthermore, $O$ is an $(\mathcal{F}_\infty \otimes \mathcal{B}(X))$-measurable backward-invariant set under the dynamical system $(\Theta^t)$. If $\varphi$ is an open mapping RDS then $O$ is also forward-invariant under $(\Theta^t)$.
\end{lemma}

\begin{proof}
Fix an open set $U \subset X$. Let $S \subset U$ be a countable set that is dense in $U$, and let $D$ be a countable dense subset of $\mathbb{T}^+$. Then we can write
\[ E_U \ = \ \bigcap_{i=1}^\infty \, \bigcup_{j=1}^\infty \, \bigcap_{t \in D \cap [j,\infty)} \, \bigcap_{x,y \in S} \, \{ \omega \in \Omega \, : \, d(\varphi(t,\omega)x,\varphi(t,\omega)y) < \tfrac{1}{i} \}. \]
\noindent So $E_U \in \mathcal{F}_\infty$. Now given a countable base $\mathcal{U}$ for the topology on $X$, one can easily check that $\,O \, = \, \bigcup_{V \in \, \mathcal{U}} E_V \times V$. Hence $O \in \mathcal{F}_\infty \otimes \mathcal{B}(X)$. By continuity, it is clear that $O$ is backward-invariant under $(\Theta^t)$. Similarly, it is clear that if $\varphi$ is an open mapping RDS then $O$ is forward-invariant under $(\Theta^t)$.
\end{proof}

\begin{defi}
For each $x \in X$, let $P_0(x)=\mathbb{P}(O_x)$ and let $P_r(x)=\mathbb{P}(E_{B_r(x)})$ for all $r>0$.
\end{defi}

\noindent It is clear that $P_r(x)$ is decreasing in $r$, with $P_0(x)=\sup_{r>0} P_r(x)=\lim_{r \to 0} P_r(x)$. Also note that the map $x \mapsto P_0(x)$ is measurable.\footnote{Indeed, one of the preparations for the statement of Fubini's theorem is that the measure of a section of a measurable set depends measurably on where the section is taken.}

\begin{lemma}
For any $x \in X$ and $t \in \mathbb{T}^+$, $P_0(x) \, \geq \, \int_X P_0(y)\,\varphi_x^t(dy)$. If $\varphi$ is an open mapping RDS then the inequality becomes equality.
\end{lemma}

\begin{proof}
Recall that $P_0(y)=\mathbb{P}(O_y)=\mathbb{P}(\theta^{-t}(O_y))$ for all $y$ and $t$. Now fix any $x \in X$ and $t \in \mathbb{T}^+$.
\begin{align*}
\int_X P_0(y) \, \varphi_x^t(dy) \ &= \ \int_\Omega P_0(\varphi(t,\omega)x) \, \mathbb{P}(d\omega) \\
&= \ \int_\Omega \mathbb{P}(\theta^{-t}(O_{\varphi(t,\omega)x})) \, \mathbb{P}(d\omega) \\
&= \ \int_\Omega \int_\Omega \mathbbm{1}_O(\theta^t\tilde{\omega},\varphi(t,\omega)x) \, \mathbb{P}(d\tilde{\omega}) \, \mathbb{P}(d\omega) \\
&= \ \int_\Omega \mathbbm{1}_O(\theta^t\omega,\varphi(t,\omega)x) \, \mathbb{P}(d\omega) \hspace{4mm} \textrm{(by the memoryless property)} \\
&\leq \ \int_\Omega \mathbbm{1}_O(\omega,x) \, \mathbb{P}(d\omega) \hspace{4mm} \textrm{(since $O$ is $(\Theta^t)$-backward-invariant)} \\
&= \ P_0(x).
\end{align*}
\noindent (For a justification of the antepenultimate line, see Exercise~124(A) of [New15].) If $\varphi$ is an open mapping RDS then $O$ is both backward- and forward-invariant under $(\Theta^t)$, so the ``$\leq$'' in the penultimate line becomes ``=''.
\end{proof}

\begin{defi}
We will say that \emph{$x$ is almost surely stable} if $P_0(x) = 1$. We will say that \emph{$x$ is potentially stable} if $P_0(x)>0$.
\end{defi}

\noindent Note that $x$ is potentially stable if and only if there is a neighbourhood $U$ of $x$ such that $\mathbb{P}(E_U)>0$.

\begin{rmk}
Given a compact set $K \subset X$, if every point in $K$ is almost surely stable, then (by the Lebesgue number lemma) $P_r(\cdot) \to 1$ uniformly on $K$ as $r \to 0$.
\end{rmk}

\begin{prop}
$\mathbb{P}$-almost every $\omega \in \Omega$ has the property that for any $x \in X$, if $(\omega,x)$ is asymptotically stable then $x$ is potentially stable.
\end{prop}

\begin{proof}
Let $\mathcal{U}$ be a countable base for the topology on $X$, and let
\[ \mathcal{U}_0 \ := \ \{U \in \mathcal{U} \, : \, \mathbb{P}(E_U) = 0 \}. \]
\noindent Let
\[ \tilde{\Omega} \ := \ \Omega \setminus \bigcup_{U \in \, \mathcal{U}_0} E_U. \]
\noindent Now fix any $\omega \in \tilde{\Omega}$ and $x \in X$. If $(\omega,x)$ is asymptotically stable then there exists $U \in \mathcal{U}$ with $x \in U$ such that $\omega \in E_U$, and hence $U \! \nin \mathcal{U}_0$; so $\mathbb{P}(E_U)>0$ and therefore $x$ is potentially stable.
\end{proof}

\begin{prop}
The set $U_\mathit{\!ps} \subset X$ of potentially stable points is an open backward-invariant set. Letting $A_s \in \mathcal{B}(X)$ denote the set of almost surely stable points, we have that $\rho(U_\mathit{\!ps} \setminus A_s)=0$ for any stationary probability measure $\rho$. If $\varphi$ is an open mapping RDS, then $\varphi_x^t(A_s)=1$ for all $x \in A_s$ and $t \in \mathbb{T}^+$.
\end{prop}

\begin{proof}
For any $x \in X$ and $r>0$ with $P_r(x)>0$, we clearly have that $B_r(x) \subset U_\mathit{\!ps}$. So $U_\mathit{\!ps}$ is open;  Lemma~2.2.5 then gives that $U_\mathit{\!ps}$ is backward-invariant. For any ergodic probability measure $\rho'$, either $\mathbb{P} \otimes \rho'(O)=0$ or $\mathbb{P} \otimes \rho'(O)=1$; in the former case, $\rho'(U_\mathit{\!ps})=\rho'(A_s)=0$, and in the latter case, $\rho'(U_\mathit{\!ps})=\rho'(A_s)=1$. So $\rho'(U_\mathit{\!ps} \setminus A_s)=0$ for every ergodic probability measure $\rho'$, and hence by Lemma~1.3.3, $\rho(U_\mathit{\!ps} \setminus A_s)=0$ for every stationary probability measure $\rho$. If $\varphi$ is an open mapping RDS then Lemma~2.2.5 gives that $\varphi_x^t(A_s)=1$ for all $x \in A_s$ and $t \in \mathbb{T}^+$.
\end{proof}

\noindent We now go on to consider ``sets admitting stable trajectories''.

\begin{lemma}
For any $A \subset X$, the set
\[ O_A \ := \ \{ \omega \in \Omega \, : \, \exists \, x \in A \ \textrm{\emph{s.t.}} \ (\omega,x) \textrm{\emph{ is asymptotically stable}} \}  \ = \ \bigcup_{x \in A} O_x \]
\noindent is $\mathcal{F}_\infty$-measurable.
\end{lemma}

\begin{proof}
Let $\mathcal{U}$ be a countable base for the topology on $X$, and let
\[ \mathcal{U}_A \ := \ \{U \in \mathcal{U} \, : \, U \cap A \neq \emptyset\}. \]
\noindent It is clear that
\[ O_A \ = \ \bigcup_{U \in \, \mathcal{U}_A} E_U, \]
\noindent giving the result.
\end{proof}

\begin{defi}
We will say that a closed forward-invariant set $K \subset X$ \emph{admits stable trajectories} if $\,\mathbb{P}(O_K) > 0$.
\end{defi}

\noindent It is not hard to show that if $\varphi$ is an \emph{open mapping} RDS then for any closed forward-invariant set $K \subset X$, $O_K$ is $\mathbb{P}$-almost invariant, and hence $K$ admits stable trajectories if and only if $\,\mathbb{P}(O_K) = 1$.

\begin{rmk}[Sufficient test for stable trajectories]
Suppose either that (a)~$X$ is a compact Riemannian manifold and $\varphi$ is a smooth RDS on $X$, or that (b)~$X$ is a compact geodesically convex subset of a Riemannian manifold $\tilde{X}$, with $\varphi$ being the $X$-restriction of some smooth RDS $\tilde{\varphi}$ on $\tilde{X}$ such that $\tilde{\varphi}(t,\omega)X \subset X$ for all $t$ and $\omega$. Let $\rho$ be an ergodic probability measure. Provided the sizes of the spatial partial derivatives of $\varphi(t,\omega)$ at $x$ are sufficiently well controlled in $(t,\omega,x)$ (over bounded ranges of $t$), we have the following test for stability: if there exists $t \in \mathbb{T}^+$ such that
\[ \Lambda_t^\rho := \int_{\Omega \times X} \! \log|| \mathrm{d}\varphi(t,\omega)_x || \, \mathbb{P} \otimes \rho(d(\omega,x)) \ < \ 0 \]
\noindent then $\rho$-almost every point in $X$ is almost surely (exponentially) stable. (This is equivalent to saying that the ``maximal Lyapunov exponent'', given by $\lim_{t \to \infty}\frac{1}{t} \log|| \mathrm{d}\varphi(t,\omega)_x ||$ for $(\mathbb{P} \otimes \rho)$-almost any $(\omega,x)$, is negative; for if $\Lambda_t^\rho<0$ for some $t$ then $\Lambda_u^\rho<0$ for all $u \geq t$.) Precise conditions for the test can be found in [LeJan87] (specifically condition (H2)\footnote{As in footnote~12, beware that in [LeJan87] the characters $||^2$ are missing from the end of the denominator in the formula for $\delta_2(T)$.} at the start of Section~3, with the proof\footnote{It is worth emphasising that, although [LeJan87] works with random \emph{diffeomorphisms}, the first of the two proofs given for Lemme~3 remains completely valid for non-invertible random smooth maps.} given in Lemme~3) in the case of discrete time; or for SDEs, see [Car85]\footnote{Note that in [Car85], the term ``Lyapunov stability'' is used to refer to strict negativity of all Lyapunov exponents.} (in particular, Proposition~2.2.3). Now if we did \emph{not} assume $\rho$ to be ergodic but only stationary, the conclusion of the test would no longer be that $\rho$-almost every point is almost surely stable---rather: given $t \in \mathbb{T}^+$ such that $\Lambda_t^\rho<0$, there will exist (by Lemma~1.3.3) an ergodic probability measure $\rho'$ such that $\rho'(\mathrm{supp}\,\rho)=1$ and $\Lambda_t^{\rho'}<0$, and so we conclude that $\mathrm{supp}\,\rho$ contains at least one point that is almost surely stable.
\end{rmk}

\begin{lemma}
A closed forward-invariant set $K$ admits stable trajectories if and only if there is a point in $K$ that is potentially stable. If $K$ is minimal then $K$ admits stable trajectories if and only if there is a point in $K$ that is almost surely stable.
\end{lemma}

\noindent We should say immediately that in the case where $K$ is minimal, a much stronger statement can be made (Theorem~2.2.14), but we will need to prove our above weaker statement first.

\begin{proof}
The first assertion follows immediately from Proposition~2.2.8. Now suppose that $K$ is minimal and admits stable trajectories. By Corollary~1.3.5 there exists an ergodic probability measure $\rho$ with $\mathrm{supp}\,\rho=K$. Since there is at least one potentially stable point in $K$ and the set of all potentially stable points is an open backward-invariant set (Proposition~2.2.9), it follows (by the minimality of $K$) that every point in $K$ is potentially stable. Hence $\mathbb{P} \otimes \rho(O)>0$, and therefore (by the $(\Theta^t)$-backward-invariance of $O$, as proved in Lemma~2.2.3) $\mathbb{P} \otimes \rho(O)=1$. So $\rho$-almost every point is almost surely stable.
\end{proof}

\begin{thm}
If $K$ is a minimal set admitting stable trajectories, then every point in $K$ is almost surely stable. If, in addition, $K$ is the only minimal set, then every point in $X$ is almost surely stable, and for each $x \in X$,
\[ \mathbb{P}(\omega : d(\varphi(t,\omega)x,K) \to 0 \textrm{ as } t \to \infty) \ = \ 1. \]
\end{thm}

\noindent In the proof of Theorem~2.2.14 (and also in the next section), we will use the following elementary lemma (which, heuristically, will play the role of the strong Markov property) in conjunction with Corollary~1.2.9 (which, heuristically, will generate a random time at which to apply the strong Markov property):

\begin{lemma}
Let $(D,\leq)$ be a countable totally ordered set. Suppose we have, for each $s \in D$ and $n \in \mathbb{N}$, events $\,R_{n,s\,} , S_{n,s} \in \mathcal{F}$ with the following properties:
\begin{itemize}
\item for all $n$ and $s$, $S_{n,s}$ is independent of $\sigma(R_{n,t}:t \leq s)$;
\item for all $n$, $\mathbb{P}(\bigcup_{s \in D} R_{n,s}) \ = \ 1$;
\item $\inf_{s \in D} \mathbb{P}(S_{n,s}) \, \to \, 1$ as $n \to \infty$.
\end{itemize}
\noindent Then
\[ \mathbb{P}\left( \bigcup_{n=1}^\infty \bigcup_{s \in D} R_{n,s} \cap S_{n,s} \right) \ = \ 1. \]
\end{lemma}

\begin{proof}
First fix $n \in \mathbb{N}$. Since $\mathbb{P}(\bigcup_{s \in D} R_{n,s}) = 1$ and $D$ is countable, we must have that for all $\varepsilon>0$ there exist $t_1 < \ldots < t_m$ in $D$ such that $\mathbb{P}(\bigcup_{i=1}^m R_{n,t_i}) > 1-\varepsilon$, and so
\begin{align*}
\mathbb{P}\left( \bigcup_{s \in D} R_{n,s} \cap S_{n,s} \right) \ &\geq \mathbb{P}\left( \bigcup_{i=1}^m R_{n,t_i} \cap S_{n,t_i} \right) \\
&\geq \ \sum_{i=1}^m \, \mathbb{P}\left(R_{n,t_i} \setminus \bigcup_{j=1}^{i-1} R_{n,t_j} \right)\mathbb{P}(S_{n,t_i}) \\
&\geq \ \left( \sum_{i=1}^m \, \mathbb{P}\left(R_{n,t_i} \setminus \bigcup_{j=1}^{i-1} R_{n,t_j} \right) \right) \inf_{s \in D} \, \mathbb{P}(S_{n,s}) \\
&= \ \mathbb{P}\left( \bigcup_{i=1}^m R_{n,t_i} \right) \inf_{s \in D} \, \mathbb{P}(S_{n,s}) \\
&\geq \ (1-\varepsilon)\inf_{s \in D} \, \mathbb{P}(S_{n,s}).
\end{align*}
\noindent This is true for all $\varepsilon$, and so
\[ \mathbb{P}\left( \bigcup_{s \in D} R_{n,s} \cap S_{n,s} \right) \ \geq \ \inf_{s \in D} \, \mathbb{P}(S_{n,s}). \]
\noindent The desired result then follows from the fact that $\inf_{s \in D} \mathbb{P}(S_{n,s}) \, \to \, 1$ as $n \to \infty$.
\end{proof}

\begin{proof}[Proof of Theorem~2.2.14]
Let $p \in K$ be an almost surely stable point. Fix any $x \in K$, and for each $n \in \mathbb{N}$ and $s \in \mathbb{Q} \cap \mathbb{T}^+$ let
\begin{align*}
R_{n,s} \ &= \ \left\{ \omega \in \Omega \, : \, \varphi(s,\omega)x \in B_{\frac{1}{n}}(p) \right\} \\
S_{n,s} \ &= \ \theta^{-s}\left(E_{B_{\!\frac{1}{n}\!}(p)}\right).
\end{align*}
\noindent Corollary~1.2.9 gives that $\mathbb{P}\left(\bigcup_s R_{n,s}\right)=1$ for all $n$. Obviously $\mathbb{P}(S_{n,s})=P_{\frac{1}{n}}(p)$ for all $n$ (independently of $s$), and so $\mathbb{P}(S_{n,s}) \to 1$ as $n \to \infty$ (uniformly in $s$). It is clear that $R_{n,s} \cap S_{n,s} \subset O_x$ for all $n$ and $s$, and so Lemma~2.2.15 yields that $\mathbb{P}(O_x)=1$, i.e. $x$ is almost surely stable.
\\ \\
Now if $K$ is the only minimal set, then (by the second assertion in Corollary~1.2.9) in the above we can take any $x \in X$ (rather than just $x \in K$). Moreover, given any $n$ and $s$, for any $\omega \in R_{n,s} \cap S_{n,s}$ we have that $d(\varphi(t,\omega)x,\varphi(t-s,\theta^s\omega)p) \to 0$ as $t \to \infty$; also, for $\mathbb{P}$-almost every $\omega \in \Omega$ we have that $\varphi(t-s,\theta^s\omega)p \in K$ for any (rational) $s$ and all $t \geq s$. Hence Lemma~2.2.15 yields that for $\mathbb{P}$-almost all $\omega$, $d(\varphi(t,\omega)x,K) \to 0$ as $t \to \infty$.
\end{proof}

\begin{rmk}
Note that, as a consequence of Theorem~2.2.14, if $\varphi$ is uniquely ergodic and $\mathbb{P} \otimes \rho(O)=1$ (where $\rho$ is the unique stationary probability measure), then every point in $X$ is almost surely stable.
\end{rmk}

\subsection{Stable synchronisation}

\begin{defi}
We will say that $\varphi$ is (\emph{globally}) \emph{stably synchronising}, or \emph{asymptotically stable in the large}, if $\varphi$ is synchronising and every point in $X$ is almost surely stable.
\end{defi}

\noindent Note once again that, due to the compactness of $X$, this is independent of the metrisation $d$ of the topology on $X$. Also note (as mentioned in Section~0.2.2) that by Remark~2.2.7, this definition is equivalent to the one given in Section~0.1.4. Moreover (as mentioned in Section~0.2.1), by Remark~2.2.16, $\varphi$ is stably synchronising if and only if $\varphi$ is synchronising and $\mathbb{P} \otimes \rho(O)=1$ (where $\rho$ is the unique stationary probability measure).

\begin{rmk}[Relation to invariant measures and random fixed points]
Suppose that for all $t \in \mathbb{T}^+$, $\theta^t$ is $\mathcal{F}$-measurably invertible. Let $(\mu_\omega)_{\omega \in \Omega}$ be a random probability measure on $X$. $(\mu_\omega)$ is said to be an \emph{invariant (random probability) measure} if for each $t \in \mathbb{T}^+$, for $\mathbb{P}$-almost all $\omega$, $\varphi(t,\omega)_\ast\mu_\omega = \mu_{\theta^t\omega}$. If, for some random variable $a:\Omega \to X$, the random probability measure $(\delta_{a(\omega)})$ is an invariant measure, then we refer to $a$ as a \emph{(random) fixed point} of $\varphi$. Now a random probability measure $(\mu_\omega)$ on $X$ (resp.~a random variable $b:\Omega \to X$) is said to be \emph{past-measurable} if the map $\omega \mapsto \mu_\omega$ (resp.~$\omega \to b(\omega)$) is measurable with respect to $\sigma(\theta^t\mathcal{F}_t:t \in \mathbb{T}^+)$. It is well-known (see e.g.~Lemme~1(a) of [LeJan87] or Theorem~4.2.9(ii) of [KS12]) that for any stationary probability measure $\rho$ there exists an associated past-measurable invariant measure $(\mu_\omega^\rho)$ given by
\[ \mu_\omega^\rho \ \overset{\mathbb{P}\textrm{-a.s.}(\omega)}{=} \ \lim_{n \to \infty} \varphi(t_n,\theta^{-t_n}\omega)_\ast\rho \]
\noindent where $(t_n)_{n \in \mathbb{N}}$ may be any unbounded increasing sequence in $\mathbb{T}^+$ and the limit is taken in the narrow topology.\footnote{The proof of Theorem~4.2.9(ii) of [KS12] does not exclude the possibility that the exceptional $\mathbb{P}$-null set depends on the sequence $(t_n)$. However, one can show that if the stochastic process $\left( \int_X g(\varphi(t,\theta^{-t\,}\boldsymbol{\cdot})x) \, \rho(dx) \right)_{\!t \geq 0}$ is a separable stochastic process for every continuous $g:X \to \mathbb{R}$ (e.g.~if the map $t \mapsto \varphi(t,\theta^{-t}\omega)x$ is left-continuous for all $x \in X$ and $\omega \in \Omega$), then the exceptional set does not depend on $(t_n)$; in other words, in this case, we can replace ``$\underset{n \to \infty}{\lim}\ldots(t_n)$'' with ``$\underset{t \to \infty}{\lim}\ldots(t)$''.} Unless $\rho$ is ``too unstable'' (in the heuristic sense that trajectories within $\mathrm{supp}\,\rho$ can escape well away from $\mathrm{supp}\,\rho$ by small perturbations), we can (heuristically) regard $\,\mathrm{supp}\,\mu_\omega^\rho\,$ as a kind of ``random attractor''; in the particular case that $\mu_\omega^\rho=\delta_{a(\omega)}$ for some random fixed point $a:\Omega \to X$, we may refer to $a$ as an ``attracting random fixed point'' or ``random point attractor''. (There are various more precise definitions of ``random attractors'' and ``random point attractors''; see e.g.~[AO03] or [Crau01]). Now it is not hard to show that if $\varphi$ is synchronising then the past-measurable invariant measure $\mu_\omega$ associated to the unique stationary probability measure must be a Dirac mass $\mathbb{P}$-almost surely---i.e.~it must be supported on a random fixed point. So if $\varphi$ is stably synchronising, then we can regard this fixed point as a random point attractor. (Indeed, this ``random point attractor'' will be an attractor in the rather strong sense that it is forward-time asymptotically stable and its basin of attraction includes a dense open set.)
\end{rmk}

\section{Necessary and Sufficient Conditions for Stable Synchronisation}

\subsection{Two-point contractibility}

\noindent For any $A \subset X$, let $\Delta_A:=\{(x,x): x \in A\} \subset X \times X$.

\begin{defi}
Given points $x,y \in X$, we will say that the pair $(x,y)$ is \emph{contractible under $\varphi$} if every neighbourhood of $\Delta_X$ is accessible from $(x,y)$ under $\varphi \! \times \! \varphi$. Given a closed~forward-invariant set $K \subset X$, we will say that \emph{$\varphi$ is two-point contractible on $K$} if for all $x,y \in K$, $(x,y)$ is contractible under $\varphi$.
\end{defi}

\noindent Note that since $\Delta_X$ is compact, any neighbourhood of $\Delta_X$ is in fact a uniform neighbourhood, and hence we have that $(x,y)$ is contractible under $\varphi$ if and only if for all $\varepsilon>0$,
\[ \mathbb{P}( \, \omega \, : \, \exists \, t \in \mathbb{T}^+ \, \textrm{ s.t. } d(\varphi(t,\omega)x,\varphi(t,\omega)y) \, < \, \varepsilon \, ) \ > \ 0. \]
\noindent Furthermore, it is easy to show (using Lemma~1.2.3 applied to $\varphi \! \times \! \varphi$) that for any closed forward-invariant $K \subset X$, $\varphi$ is two-point contractible on $K$ if and only if for all $x,y \in K$ with $x \neq y$,
\[ \mathbb{P}( \, \omega \, : \, \exists \, t \in \mathbb{T}^+ \, \textrm{ s.t. } d(\varphi(t,\omega)x,\varphi(t,\omega)y) \, < \, d(x,y) \, ) \ > \ 0. \]
\noindent This is, in turn, equivalent to saying that there are no non-empty forward-invariant compact subsets of $(K \times K) \setminus \Delta_K$ (which is essentially the formulation given in [BS88] and [Bax91]).
\\ \\
Take a sequence of values $\varepsilon_n$ decreasing to 0; if we apply Lemma~1.2.7 to the two-point motion $\varphi \! \times \! \varphi$, replacing $K$ with $\{(u,v) \in K \times K : d(u,v) \geq \varepsilon_n \}$, we obtain the following important fact: \emph{If $\varphi$ is two-point contractible on $K$ then for every $x,y \in K$,}
\[ \mathbb{P}( \, \omega \, : \, \exists \, \textrm{unbounded } (t_n)_{n \in \mathbb{N}} \textrm{ in } \mathbb{T}^+ \textrm{ s.t. } d(\varphi(t_n,\omega)x,\varphi(t_n,\omega)y) \to 0 \textrm{ as } n \to \infty \, ) \ = \ 1. \]

\noindent (See also Proposition~4.1 of [BS88].)
\\ \\
It is worth observing that if $\varphi$ is two-point contractible on the whole of $X$ then any two non-empty closed forward-invariant subsets of $X$ must have non-trivial intersection, and so there must be a unique minimal set.
\\ \\ 
Now there may be contexts in which it is not directly clear that $\varphi$ is two-point contractible on the whole of a minimal set $K$, and yet where it is directly clear that (for some given stationary probability measure supported by $K$) $\varphi$ is two-point contractible on \emph{almost} the whole of $K$ (e.g.~when, for some $t$, the support of the $C(X,X)$-valued random variable $\omega \mapsto \varphi(t,\omega)$ includes a map with a fixed point in $K$ whose basin of attraction includes all but a zero-measure set in $K$; see, for example, Theorem~1.1 of [Hom13]). For such cases, we have the following:

\begin{prop}
Suppose $\rho$ is a stationary probability measure with $K:=\mathrm{supp}\,\rho$ being minimal. Suppose also that there is a set $A \subset K$ such that
\begin{itemize}
\item $\rho(A)=1$;
\item the interior of $A$ relative to $K$ is non-empty;
\item every pair of points in $A$ is contractible under $\varphi$.
\end{itemize}
\noindent Then $\varphi$ is two-point contractible on $K$.
\end{prop}

\begin{proof}
Fix $x,y \in K$; it is sufficient to show that $G_{(x,y)} \cap (A \times A) \neq \emptyset$. Let $D$ be a countable dense subset of $\mathbb{T}^+$. Let
\[ B \ := \ \{ x \in K \, : \, \textrm{for all }t \in D, \ \varphi_x^t(A)=1\}. \]
\noindent Since $\rho$ is stationary and $\rho(A)=1$, it is easy to show that $\rho(B)=1$---and so, in particular, $B$ is non-empty. By Lemma~1.2.4 the two projections of $G_{(x,y)}$ are both equal to $K$; so let $(u,v)$ be a point in $G_{(x,y)}$ such that $u \in B$. Let $U \subset X$ be an open set such that $U \cap K$ is a non-empty subset of $A$; so $U$ is acessible from $v$. Since the map $t \mapsto \varphi_v^t(U)$ is right lower semicontinuous, there must exist $t \in D$ such that $\varphi_v^t(U)>0$. Since $K$ is forward-invariant, it follows that $\varphi_v^t(A)>0$. So then, there exists a $\mathbb{P}$-positive measure set of sample points $\omega$ such that $\varphi(t,\omega)u$ and $\varphi(t,\omega)v$ are both in $A$. So $G_{(u,v)}$ has non-trivial intersection with $A \times A$, and hence $G_{(x,y)}$ has non-trivial intersection with $A \times A$.
\end{proof}

\subsection{Main result}

\begin{thm}
$\varphi$ is stably synchronising if and only if the following conditions hold:
\begin{enumerate}[\indent (i)]
\item there is a unique minimal set $K \subset X$;
\item $\varphi$ is two-point-contractible on the minimal set $K$;
\item the minimal set $K$ admits stable trajectories.
\end{enumerate}
\end{thm}

\begin{rmk}
If we are given a stationary probability measure $\rho$, then we can replace condition (i) with the condition that $\mathrm{supp}\,\rho$ is the unique minimal set; for conditions (ii) and (iii), we then set $K:=\mathrm{supp}\,\rho$.
\end{rmk}

\begin{rmk}
It is not hard to show, using the Poincar\'{e} recurrence theorem, that if $\varphi$ satisfies (i) and (ii) (e.g.~if $\varphi$ is two-point contractible on the whole of $X$) and $K$ has non-empty interior, then the support of every stationary probability measure is equal to $K$.
\end{rmk}

\begin{proof}[Proof of Theorem 3.2.1]
If $\varphi$ is synchronising then it is obviously two-point contractible on $X$, so (i) and (ii) hold. If $\varphi$ is stably synchronising then (iii) also holds. Now suppose that (i), (ii) and (iii) hold. By Theorem~2.2.14, (i) and (iii) together imply that every point in $X$ is almost surely stable, and so we just need to establish that $\varphi$ is synchronising. Using the second assertion in Theorem~2.2.14, it is clear that for any closed non-empty forward-invariant $C \subset X \times X$, $C$ has non-trivial intersection with $K \times K$; consequently, by (iii), $C$ has non-trivial intersection with $\Delta_{K}$---and therefore, since $\Delta_K$ is clearly minimal under $\varphi \! \times \! \varphi$, $C$ contains the whole of $\Delta_K$. Thus we see that $\Delta_K$ is the unique minimal set under $\varphi \! \times \! \varphi$. Now fix any $x,y \in X$. Fix a point $p \in K$, and for each $n \in \mathbb{N}$ and $s \in \mathbb{Q} \cap \mathbb{T}^+$, let
\begin{align*}
R_{n,s} \ &= \ \left\{ \omega \in \Omega \, : \, (\varphi(s,\omega)x,\varphi(s,\omega)y) \, \in B_{\frac{1}{n}}(p) \times B_{\frac{1}{n}}(p) \right\} \\
S_{n,s} \ &= \ \theta^{-s}\left(E_{B_{\!\frac{1}{n}\!}(p)}\right).
\end{align*}
\noindent Corollary~1.2.9 (applied to $\varphi \! \times \! \varphi$) gives that $\mathbb{P}\left(\bigcup_s R_{n,s}\right)=1$ for all $n$. Since $p$ is almost surely stable, we have once again that $\mathbb{P}(S_{n,s}) \to 1$ as $n \to \infty$ (uniformly in $s$). Obviously, given any $n$ and $s$, for any $\omega \in R_{n,s} \cap S_{n,s}$ we have that $d(\varphi(t,\omega)x,\varphi(t,\omega)y) \to 0$ as $t \to \infty$. So Lemma~2.2.15 yields that for $\mathbb{P}$-almost all $\omega$, $d(\varphi(t,\omega)x,\varphi(t,\omega)y) \to 0$ as $t \to \infty$. So $\varphi$ is synchronising.
\end{proof}

\section{Example}

The following is based on the example described in [LeJan87].\footnote{A related study (a preliminary version of which is cited in [LeJan87]) can also be found in [Kai93]; it is from this paper that we have taken our term ``subperiod''.} Let $(\Omega,\mathcal{F},\mathbb{P})=([0,1)^\mathbb{N},\mathcal{B}([0,1))^{\otimes \mathbb{N}},\nu^{\otimes \mathbb{N}})$, where $\nu$ is the Lebesgue measure on $[0,1)$. For each $n \in \mathbb{N} \cup \{0\}$, define the projection $c_n:\Omega \to [0,1)^n$ by $c_n((\alpha_r)_{r \in \mathbb{N}})=(\alpha_r)_{r \in \{1,\ldots,n\}}$, and let $\mathcal{F}_n:=\sigma(c_n)$. Writing $\theta^n:(\alpha_r)_{r \in \mathbb{N}} \mapsto (\alpha_{n+r})_{r \in \mathbb{N}}$ for the canonical shift dynamical system, it is clear that $(\Omega,\mathcal{F},(\mathcal{F}_n),\mathbb{P},(\theta^n))$ is a memoryless noise space.
\\ \\
We identify the circle $\mathbb{S}^1$ with $^{\mathbb{R}\!}/_{\!\mathbb{Z}\,}$ in the obvious manner (where $^{\mathbb{R}\!}/_{\!\mathbb{Z}\,}$ is equipped with the obvious topology and Riemannian structure), and we write $[x]:=\{x+n:n \in \mathbb{Z}\} \in \mathbb{S}^1$ for all $x \in \mathbb{R}$. Let $f:\mathbb{S}^1 \to \mathbb{S}^1$ be a smooth map, and $F:\mathbb{R} \to \mathbb{R}$ a lift of $f$, such that $F(x)-x$ is periodic in $x$ (i.e.~$\mathrm{deg}\,f=1$). For any $\alpha \in [0,1)$, we define the map $f_\alpha:\mathbb{S}^1 \to \mathbb{S}^1$ by
\[ f_\alpha([x]) \ = \ [F(x+\alpha) - \alpha]. \]
\noindent We will say that a value $\alpha \in (0,1)$ is a \emph{subperiod of $f$} if $f_\alpha = f$. (This is equivalent to saying that $F(x)-x$ is $\alpha$-periodic in $x$.) It is clear that for any $\alpha_1,\alpha_2 \in (0,1)$, $\alpha_1$ is a subperiod of $f$ if and only if it is a subperiod of $f_{\alpha_2}$.
\\ \\
We define the RDS $\varphi^f$ on $\mathbb{S}^1$ by
\[ \varphi^f(n,(\alpha_r)) \ = \ f_{\alpha_n} \circ \ldots \circ f_{\alpha_1}. \]

\begin{p1}
$\mathbb{S}^1$ is minimal under $\varphi^f$ if and only if $f$ is not a rational rotation.
\end{p1}

\begin{proof}
If $f$ is a rotation then $\varphi^f(n,\omega)=f^n$ for all $n$ and $\omega$; hence it is clear that if $f$ is a rational rotation then $\varphi^f$ is not minimal on $\mathbb{S}^1$. Now suppose that $f$ is not a rational rotation; so there must exist $p \in \mathbb{R}$ such that $w \! := \! F(p)-p$ is irrational. For all $y \in \mathbb{S}^1$ let $\alpha(y) \in [0,1)$ be such that $[p-\alpha(y)]=y$. Fix $x \in \mathbb{S}^1$, with $x' \in \mathbb{R}$ being a lift of $x$, and any non-empty open $U \subset \mathbb{S}^1$. Let $n \in \mathbb{N}$ be such that $[x' + nw] \in U$. Let $\alpha_1:=\alpha(x)$, and recursively define $\alpha_m:=\alpha(f_{\alpha_{m-1}} \circ \ldots \circ f_{\alpha_1}(x))\,$ for $2 \leq m \leq n$. Then by construction, $f_{\alpha_n} \circ \ldots \circ f_{\alpha_1}(x) \in U$. It is then clear that for a sufficiently small neighbourhood $V$ of $(\alpha_1,\ldots,\alpha_n)$ in $[0,1)^n$, $\varphi^f(n,\omega)x \in U$ for all $\omega \in c_n^{-1}(V)$. Since $\nu$ has full support on $[0,1)$, it follows that $\mathbb{P}(c_n^{-1}(V))>0$. Thus $U$ is accessible from $x$. This is true for all $x$ and $U$, so $\varphi^f$ is minimal.
\end{proof}

\begin{p2}
$\varphi^f$ is two-point contractible on $\mathbb{S}^1$ if and only if $f$ has no subperiods.
\end{p2}

\begin{proof}
We work with the metric $d(x,y)=\min\{|y'-x'|:[x']=x,[y']=y\}$ (under which the diameter of $\mathbb{S}^1$ is $\frac{1}{2}$). First suppose $f$ has a subperiod; without loss of generality, let $\alpha \in (0,\frac{1}{2}]$ be a subperiod. Since $\alpha$ is a subperiod of $f_{\alpha'}$ for all $\alpha' \in [0,1)$, it follows that for any $x,y \in \mathbb{S}^1$ with $d(x,y)=\alpha$, $d(\varphi^f(n,\omega)x,\varphi^f(n,\omega)y)=\alpha$ for all $n$ and $\omega$. So $\varphi^f$ is not two-point contractible on $\mathbb{S}^1$.
\\ \\
Now suppose $f$ has no subperiods. Fix any $\alpha \in (0,\frac{1}{2}]$, and define the function $g:\mathbb{R} \to \mathbb{R}$ by $g(z)=F(z+\alpha)-F(z)-\alpha$. Since $\alpha$ is not a subperiod, there exists $q \in \mathbb{R}$ such that $g(q) \neq 0$. We now consider separately the possibilities that $g(q)>0$ and $g(q)<0$. First suppose $g(q)>0$. Given the continuity of $F$, let $\varepsilon \in (0,1)$ be such that for all $x \in [q,q+\varepsilon)$, $F(x)-F(q) - (x-q)<g(q)$. Let $r \geq 2$ be an integer such that $\langle r\alpha \rangle \in [0,\varepsilon)$ (where $\langle a \rangle:=a-\lfloor a \rfloor$ denotes the fractional part of $a$). We have that
\begin{align*} \sum_{i=0}^{r-1} g(q+i\alpha) \ &= \ F(q+r\alpha) - F(q) - r\alpha \\ &= \ F(q+ \langle r\alpha \rangle) - F(q) - \langle r\alpha \rangle \hspace{4mm} \textrm{(since $x \mapsto F(x)-x$ is 1-periodic)} \\ &< \ g(q). \end{align*}
\noindent Hence there must exist $i \in \{1,\ldots,r-1\}$ such that $g(q+i\alpha)<0$. Now in the case that $g(q)<0$, by the same argument we can find $i \in \mathbb{N}$ such that $g(q+i\alpha)>0$. In either case, $g$ attains both positive and negative values, and so by the intermediate value theorem, the range of $g$ includes a neighbourhood of 0. So let $p \in \mathbb{R}$ be such that $g(p) \in [-\alpha,0)$. Now fix any points $x,y \in \mathbb{S}^1$ with $d(x,y)=\alpha$; without loss of generality, assume there exist lifts $x',y' \in \mathbb{R}$ of $x$ and $y$ respectively such that $y'-x'=\alpha$. Let $\beta := \langle p-x' \rangle$. Then
\begin{align*}
d(f_\beta(x),f_\beta(y)) \ &= \ d( \, [\!F(x'+\beta)] \, , \, [\!F(y'+\beta)] \, ) \\
&= \ d( \, [\!F(p)] \, , \, [\!F(p+\alpha)] \, ) \\
&< \ \alpha.
\end{align*}
\noindent Hence (by continuity) there must exist a non-empty open $U \subset [0,1)$ such that
\[ d(f_{\tilde{\beta}}(x),f_{\tilde{\beta}}(y)) \ < \ \alpha \]
\noindent for all $\tilde{\beta} \in U$. Note that $\mathbb{P}(c_1^{-1}(U)) = \nu(U)>0$. Hence $\varphi^f$ is two-point contractible on $\mathbb{S}^1$.
\end{proof}

\noindent Now define the quantity $\lambda_f \in \mathbb{R} \cup \{-\infty\}$ by
\[ \lambda_f \ := \ \int_0^1 \log |F'(y)| \, dy. \]

\begin{l3}
The Lebesgue measure $l$ on the circle is a stationary probability measure of $\varphi^f$. For any $n \in \mathbb{N} \cup \{0\}$,
\[ \int_{\Omega \times \mathbb{S}^1} \! \log|| \mathrm{d}\varphi^f(n,\omega)_x || \, \mathbb{P} \otimes l(d(\omega,x)) \ = \ n\lambda_f. \]
\end{l3}

\begin{proof}
A straightforward exercise.
\end{proof}

\begin{t4}
If $f$ has no subperiods and $\lambda_f<0$ then $\varphi^f$ is stably synchronising.
\end{t4}

\noindent Note that a partial converse also holds: if $f$ has a subperiod then clearly $\varphi^f$ cannot be synchronising.

\begin{proof}
Obviously if $f$ has no subperiods then, in particular, $f$ is not a rotation. So Propositions~4.1 and 4.2 give that $\varphi^f$ satisfies conditions (i) and (ii) of Theorem~3.2.1, with $K$ being the whole of $\mathbb{S}^1$. Now it is clear that for any $n \in \mathbb{N} \cup \{0\}$, every order derivative of $\varphi^f(n,\omega)$ is uniformly bounded in $\omega$; this means that $\varphi^f$ easily satisfies the conditions required to be able to apply the ``negative Lyapunov exponent'' rule described in Remark~2.2.12. So by Lemma~4.3, if $\lambda_f<0$ then $\mathbb{S}^1$ admits stable trajectories. Theorem~3.2.1 then gives the result.
\end{proof}

\begin{r5}
Due to Lemma~1.3.3 (as it is used in Remark~2.2.12), we did not need to prove the ergodicity of $l$ in order to prove Theorem~4.4. Nonetheless, one can show that provided $f$ is not a rational rotation, $l$ is the only stationary probability measure (and in particular is therefore ergodic): firstly, in the case that $f$ is an irrational rotation, $\varphi^f$ simply consists of repeated iterations of $f$, and it is well known that $l$ is the only invariant probability measure of an irrational rotation (see e.g.~Theorem~11.2.9 of [KH95]); secondly, in the case that $f$ is not a rotation, one can show that $l$ must be the only stationary probability measure, using the fact that for any $x \in \mathbb{S}^1$ there will exist a non-empty open set $V \subset (0,1)$ such that the map $\alpha \mapsto f_\alpha(x)$ is a local diffeomorphism on $V$.
\end{r5}

\noindent Now [LeJan87] specifically considers the case that $f$ is an orientation-preserving diffeomorphism---in which case (as is mentioned in [LeJan87]), if $f$ is not a rotation then $\lambda_f<0$ (by the strict Jensen inequality). It is stated that if $\frac{1}{n}$ is the least period of $x \mapsto F(x)-x$, then the number of open regions of mutual convergence (as described in Section~0.2) is equal to $n$. The justification (as is very briefly outlined in [LeJan87]) is essentially as follows: due to $f$ being homeomorphic, each region of mutual convergence must be a \emph{connected} set of length $\frac{1}{m}$ (where $m$ is the number of such regions); and moreover (due to the shift map $\theta$ being measure-preserving) the images of each of these sets under the RDS $\varphi^f$ must \emph{remain} of length $\frac{1}{m}$; since at each step the random parameter $\alpha$ can be selected anywhere from $[0,1)$, it follows that $m$ is a subperiod of $f$ and therefore $n$ is a multiple of $m$; but it is also clear that $m$ is a multiple of $n$.
\\ \\
So in particular, the ``$n=1$ case'' of Le~Jan's statement is that if $f$ is a diffeomorphism with no subperiods then $\varphi^f$ is ``$l$-almost stably synchronising'' (as defined in Section~0.2). Theorem~4.4 strengthens this to stable synchronisation, and extends this \emph{beyond} the case of invertible $f$.
\\ \\
As an example: Consider the map $f$ whose lift is given by
\[ F(x) \ = \ x + a\sin(2 \pi x), \hspace{6mm} a > 0. \]
\noindent If $a \in (0,\frac{1}{2\pi})$ then $f$ is an orientation-preserving diffeomorphism with no subperiods, and so $\varphi^f$ is stably synchronising. If we wish to increase $a$ beyond $\frac{1}{2\pi}$, such that $f$ will no longer be a diffeomorphism or even an injective map, there will still exist a range of $a$-values close enough to $\frac{1}{2\pi}$ from above such that $\lambda_f<0$ and so $\varphi^f$ remains stably synchronising.

\begin{r6}
Heuristically, in the example studied above, the ``noise process'' is a sequence of random rotations of the phase space on which $f$ is being iterated. We have seen that this noise process can induce global stable synchronisation when the original map $f$ has much less straightforward dynamics. Indeed (as in [LeJan87]), for any diffeomorphism $f$ with \emph{any} number of attracting, repelling and non-hyperbolic fixed points or periodic points, provided $f$ has no subperiods we will have that $\varphi^f$ is globally stably synchronising. (And even if $f$ does have a subperiod, there will obviously exist an arbitrarily small perturbation of $f$ which has no subperiods.)
\end{r6}

\section*{Appendix: Asymptotic vs. Lyapunov stability}

Asymptotic stability is usually defined as the combination of Lyapunov stability and local attractivity. However, the definition that we have been working with (Definition~2.2.1) is drastically easier to work with---partly because it is inherently a much simpler definition, and partly because the na\"{i}ve notion of ``local attractivity'' poses difficulties with regards to measurability. Nonetheless, if we are to work with a different definition from the usual definition, then it is important that our definition is at least ``similar'' to the usual definition. And above all, it is important that our definition of ``asymptotic stability'' implies Lyapunov stability (otherwise it would be absurd to refer to it as ``stability''). In this section, we shall deal with these issues.
\\ \\
(The main results in this Appendix are Theorem~A11 and Corollary~A13. Theorem~A11 is not specific to \emph{random} dynamical systems, but is really a general result about non-autonomous dynamical systems.)
\\ \\
Recall that the RDS $\varphi$ is assumed to be c\`{a}dl\`{a}g (properties (d) and (e) in Definition~1.1.8).

\begin{alemma}
For any $t \in \mathbb{T}^+$ and $\omega \in \Omega$, the family of functions $\{\varphi(s,\omega)\}_{0 \leq s \leq t}$ on $X$ is equicontinuous.
\end{alemma}

\begin{proof}
Suppose the family of maps $\{\varphi(s,\omega)\}_{0 \leq s \leq t}$ is not equicontinuous. Then there exist a number $\varepsilon>0$, a sequence $(x_n)_{n \in \mathbb{N}}$ in $X$ converging to a point $x$, and a sequence $(s_n)_{n \in \mathbb{N}}$ in $\mathbb{T}^+ \cap [0,t]$, such that $d(\varphi(s_n,\omega)x_n,\varphi(s_n,\omega)x) > \varepsilon$ for all $n$. By the monotone subsequence theorem, let $(n_i)_{i \in \mathbb{N}}$ be an unbounded increasing sequence in $\mathbb{N}$ such that \emph{either} $(s_{n_i})_{i \in \mathbb{N}}$ is decreasing (with $s:=\inf_i s_{n_i}$) \emph{or} $(s_{n_i})_{i \in \mathbb{N}}$ is strictly increasing (with $s:=\sup_i s_{n_i}$). If $(s_{n_i})$ is decreasing then (by property~(d) in Definition~1.1.8) $\varphi(s_{n_i},\omega)x_{n_i} \to \varphi(s,\omega)x$ and $\varphi(s_{n_i},\omega)x \to \varphi(s,\omega)x$ as $i \to \infty$, so $d(\varphi(s_{n_i},\omega)x_{n_i},\varphi(s_{n_i},\omega)x) \to 0$ as $i \to \infty$, contradicting the fact that $d(\varphi(s_n,\omega)x_n,\varphi(s_n,\omega)x) > \varepsilon$ for all $n$. If $(s_{n_i})$ is strictly increasing then (by property~(e) in Definition~1.1.8) $\varphi(s_{n_i},\omega)x_{n_i} \to \varphi_-(s,\omega,x)$ and $\varphi(s_{n_i},\omega)x \to \varphi_-(s,\omega,x)$ as $i \to \infty$, so once again, $d(\varphi(s_{n_i},\omega)x_{n_i},\varphi(s_{n_i},\omega)x) \to 0$ as $i \to \infty$, contradicting the fact that $d(\varphi(s_n,\omega)x_n,\varphi(s_n,\omega)x) > \varepsilon$ for all $n$. In either case, we have a contradiction.
\end{proof}

\begin{adefi}
For any $\varepsilon>0$ and any $0 < \delta \leq \varepsilon$, we will say that a pair $(\omega,x) \in \Omega \times X$ is \emph{$(\varepsilon,\delta)$-contained} if for any $y \in X$ with $d(x,y) \leq \delta$ and any $t \in \mathbb{T}^+$, $d(\varphi(t,\omega)x,\varphi(t,\omega)y) \leq \varepsilon$.
\end{adefi}

\noindent Now for any $\varepsilon>0$ and any $0 < \delta \leq \varepsilon$, let $L_{\varepsilon,\delta} \subset \Omega \times X$ denote the set of all $(\varepsilon,\delta)$-contained pairs $(\omega,x)$.

\begin{alemma}
For all $\varepsilon>0$ and $0 < \delta \leq \varepsilon$, $L_{\varepsilon,\delta} \in \mathcal{F}_\infty \otimes \mathcal{B}(X)$.
\end{alemma}

\begin{proof}
Let $S$ be a countable dense subset of $X$, and let $D$ be a countable dense subset of $\mathbb{T}^+$. It is easy to show that
\[ L_{\varepsilon,\delta} \ = \ \bigcap_{t \in D} \, \bigcap_{y \in S} \; \{ (\omega,x) : \textrm{either } d(x,y) \leq \delta \textrm{ and } d(\varphi(t,\omega)x,\varphi(t,\omega)y) \leq \varepsilon, \textrm{ or } d(x,y)>\delta \}, \]
\noindent which is clearly $(\mathcal{F}_\infty \otimes \mathcal{B}(X))$-measurable.
\end{proof}

\begin{adefi}
We will say that a pair $(\omega,x) \in \Omega \times X$ is \emph{Lyapunov stable} if the family of maps $\{\varphi(t,\omega)\}_{t \in \mathbb{T}^+}$ is equicontinuous at $x$---i.e.~if for all $\varepsilon>0$ there exists $0<\delta \leq \varepsilon$ such that $(\omega,x)$ is $(\varepsilon,\delta)$-contained.
\end{adefi}

\noindent Now let $L \subset \Omega \times X$ denote the set of all Lyapunov stable pairs $(\omega,x)$.

\begin{aprop}
$L$ is an $(\mathcal{F}_\infty \otimes \mathcal{B}(X))$-measurable $(\Theta^t)$-backward-invariant set.
\end{aprop}

\begin{proof}
It is clear that
\[ L \ = \ \bigcap_{n=1}^\infty \, \bigcup_{m=n}^\infty L_{\frac{1}{n},\frac{1}{m}}, \]
\noindent and so $L \in \mathcal{F}_\infty \otimes \mathcal{B}(X)$. Now suppose we have $t \in \mathbb{T}^+$, $\omega \in \Omega$ and $x \in X$ such that $\Theta^t(\omega,x)$ is Lyapunov stable; we need to show that $(\omega,x)$ is Lyapunov stable. Fix any $\varepsilon>0$. Let $0<\tilde{\delta} \leq \varepsilon$ be such that $\Theta^t(\omega,x)$ is $(\varepsilon,\tilde{\delta})$-contained, and (by Lemma~A1) let $\delta>0$ be such that $\varphi(s,\omega)\overline{B_\delta(x)} \subset \overline{B_{\tilde{\delta}}(\varphi(s,\omega)x)}$ for all $s \in \mathbb{T}^+ \cap [0,t]$. Then $(\omega,x)$ is $(\varepsilon,\delta)$-contained.
\end{proof}

\begin{adefi}
For any $\varepsilon>0$ and any $0 < \delta \leq \varepsilon$, we will say that a pair $(\omega,x) \in \Omega \times X$ is \emph{$(\varepsilon,\delta)$-recurrent} if the set
\[ R_{\varepsilon,\delta,\omega,x} \ := \ \{ t \in \mathbb{T}^+ : \, \textrm{$\Theta^t(\omega,x)$ is $(\varepsilon,\delta)$-contained} \} \]
\noindent is unbounded.
\end{adefi}

\begin{adefi}
We will say that a pair $(\omega,x) \in \Omega \times X$ is \emph{recurrently Lyapunov stable} if for all $\varepsilon>0$ there exists $0<\delta \leq \varepsilon$ such that $(\omega,x)$ is $(\varepsilon,\delta)$-recurrent.
\end{adefi}

\noindent Note that, by Proposition~A5, if a pair $(\omega,x)$ is recurrently Lyapunov stable then indeed it is Lyapunov stable. Let $L' \subset \Omega \times X$ denote the set of recurrently Lyapunov stable pairs $(\omega,x)$.

\begin{aprop}
For any stationary probability measure $\rho$, $L \setminus L'$ is a $(\mathbb{P} \otimes \rho)$-null set.
\end{aprop}

\begin{proof}
For any $\varepsilon>0$ and $0<\delta \leq \varepsilon$, let $L'_{\varepsilon,\delta}$ denote the set of all $(\varepsilon,\delta)$-recurrent pairs $(\omega,x)$. Then
\[ L \setminus L' \ = \ \left( \bigcap_{n=1}^\infty \, \bigcup_{m=n}^\infty L_{\frac{1}{n},\frac{1}{m}} \right) \setminus \left( \bigcap_{n=1}^\infty \, \bigcup_{m=n}^\infty L'_{\frac{1}{n},\frac{1}{m}} \right) \ \subset \ \bigcup_{n=1}^\infty \, \bigcup_{m=n}^\infty \left( L_{\frac{1}{n},\frac{1}{m}} \setminus L'_{\frac{1}{n},\frac{1}{m}} \right). \]
\noindent Now for any integers $m \geq n$, Lemma~A3 and the Poincar\'{e} recurrence theorem yield that $L_{\frac{1}{n},\frac{1}{m}} \setminus L'_{\frac{1}{n},\frac{1}{m}}$ is a $(\mathbb{P} \otimes \rho)$-null set. So we are done.
\end{proof}

\begin{armk}
It is clear that for a fixed point of an autonomous deterministic dynamical system, Lyapunov stability and recurrent Lyapunov stability are equivalent. Similarly, in the random context: if $\theta^t$ is measurably invertible for each $t$, then Proposition~A8 yields that for any past-measurable random fixed point $a:\Omega \to X$ (see Remark~2.3.2), $(\omega,a(\omega))$ is Lyapunov stable for $\mathbb{P}$-almost all $\omega$ if and only if $(\omega,a(\omega))$ is recurrently Lyapunov stable for $\mathbb{P}$-almost all $\omega$.
\end{armk}

\begin{adefi}
We will say that a pair $(\omega,x) \in \Omega \times X$ is (\emph{locally}) \emph{attractive in the na\"{i}ve sense} if there exists a neighbourhood $U$ of $x$ such that for all $y \in U$, $x$ and $y$ mutually converge under $\omega$.
\end{adefi}

\begin{athm}
(I) If a pair $(\omega,x)$ is both recurrently Lyapunov stable and attractive in the na\"{i}ve sense, then it is asymptotically stable. (II) If a pair $(\omega,x)$ is asymptotically stable, then it is both Lyapunov stable and attractive in the na\"{i}ve sense.
\end{athm}

\begin{proof}
(I) Suppose for a contradiction that $(\omega,x)$ is recurrently Lyapunov stable and attractive in the na\"{i}ve sense, but not asymptotically stable. Let $\varepsilon>0$ be such that for any $y \in \overline{B_\varepsilon(x)}$, $x$ and $y$ mutually converge under $\omega$. Since $(\omega,x)$ is not asymptotically stable, there exist a sequence $(x_n)_{n \in \mathbb{N}}$ in $B_\varepsilon(x)$, an unbounded sequence of times $(t_n)_{n \in \mathbb{N}}$ in $\mathbb{T}^+$, and a value $\varepsilon'>0$, such that $d(\varphi(t_n,\omega)x_n,\varphi(t_n,\omega)x)>\varepsilon'$ for all $n \in \mathbb{N}$. Let $(n_i)_{i \in \mathbb{N}}$ be an unbounded increasing sequence in $\mathbb{N}$ such that $x_{n_i}$ converges to a point $x_\infty$ as $i \to \infty$. Obviously $x_\infty \in \overline{B_\varepsilon(x)}$, and so $x_\infty$ and $x$ mutually converge under $\omega$. Let $0 < \delta \leq \varepsilon'$ be such that $R_{\varepsilon',\delta,\omega,x}$ is unbounded. So there exists $t^\ast \in R_{\varepsilon',\delta,\omega,x}$ such that $d(\varphi(t^\ast,\omega)x_\infty,\varphi(t^\ast,\omega)x)=:\delta' < \delta$. Let $N \in \mathbb{N}$ be such that for all $i \geq N$, $d(\varphi(t^\ast,\omega)x_{n_i},\varphi(t^\ast,\omega)x_\infty)<\delta-\delta'$; so $d(\varphi(t^\ast,\omega)x_{n_i},\varphi(t^\ast,\omega)x)<\delta$ for all $i \geq N$. Let $i^\ast \in \mathbb{N} \cap [N,\infty)$ be such that $t_{n_{i^\ast}}\!\!>t^\ast$. Since $\Theta^{t^\ast}(\omega,x)$ is $(\varepsilon',\delta)$-contained, it follows that $d(\varphi(t_{n_{i^\ast}},\omega)x_{n_{i^\ast}},\varphi(t_{n_{i^\ast}},\omega)x)<\varepsilon'$. But this contradicts the fact that $d(\varphi(t_n,\omega)x_n,\varphi(t_n,\omega)x)>\varepsilon'$ for all $n \in \mathbb{N}$.
\\ \\
(II) Suppose $(\omega,x)$ is asymptotically stable. Obviously, $(\omega,x)$ is attractive in the na\"{i}ve sense. Now let $r>0$ be such that $B_r(x)$ contracts under $\omega$, and fix any $\varepsilon>0$. Let $T \in \mathbb{T}^+$ be such that for all $t \geq T$, $\varphi(t,\omega)B_r(x) \subset B_\varepsilon(\varphi(t,\omega)x)$. On the basis of Lemma~A1, let $0<\delta<r$ be such that $\varphi(s,\omega)\overline{B_\delta(x)} \subset B_\varepsilon(\varphi(s,\omega)x)$ for all $s \in \mathbb{T}^+ \cap [0,T]$. Then $(\omega,x)$ is $(\varepsilon,\delta)$-contained.
\end{proof}

\begin{armk}
Theorem A11 does not fully rely on the compactness of $X$: Part~(I) relies on the \emph{local} compactness of $X$ (together with $\varphi$ being continuous in space). Part~(II) relies on $\varphi$ being c\`{a}dl\`{a}g (since it uses Lemma~A1), but does not rely on any properties of $X$ (beyond being a metric space).
\end{armk}

\begin{acor}
Let $\rho$ be a stationary probability measure. Then for $(\mathbb{P} \otimes \rho)$-almost every $(\omega,x) \in \Omega \times X$, $(\omega,x)$ is asymptotically stable if and only if it is both Lyapunov stable and attractive in the na\"{i}ve sense.
\end{acor}

\begin{proof}
Follows immediately from Theorem~A11 and Proposition~A8.
\end{proof}

\subsection*{Acknowledgement}

The author would like to thank Dr Martin Rasmussen, Prof Jeroen Lamb and Prof~Thomas~Kaijser for their valuable comments and suggestions (as well as their warm encouragement).

\subsection*{References}

[Pik84] Pikovskii, A.~S., Synchronization and stochastization of array of self-excited oscillators by external noise, \emph{Radiophysics and Quantum Electronics} \textbf{27}(5), pp390--395.
\\ \\
\textrm{[NANTK05]} Nakao, H., Arai, K., Nagai, K., Tsubo, Y., Kuramoto, Y., Synchrony of limit-cycle oscillators induced by random external impulses, \emph{Physical Review E} \textbf{72}, 026220.
\\ \\
\textrm{[KFI12]} Kurebayashi, W., Fujiwara, K., Ikeguchi, T., Colored noise induces synchronization of limit cycle oscillators, \emph{Europhysics Letters} \textbf{97}(5), 50009.
\\ \\
\textrm{[LP13]} Lai, Y. M., Porter, M. A., Noise-induced synchronization, desynchronization, and clustering in globally coupled nonidentical oscillators, \emph{Physical Review E} \textbf{88}, 012905.
\\ \\
\textrm{[Arn98]} Arnold, L., \emph{Random Dynamical Systems}, Springer, Berlin Heidelberg New York.
\\ \\
\textrm{[App04]} Applebaum, D., \emph{L\'{e}vy Processes and Stochastic Calculus}, Cambridge Studies in Advanced Mathematics 116, Cambridge University Press, Cambridge~New~York.
\\ \\
\textrm{[BS88]} Baxendale, P. H., Stroock, D. W., Large Deviations and Stochastic Flows of Diffeomorphisms, Probability Theory and Related Fields \textbf{80}(2), pp169--215.
\\ \\
\textrm{[Bax91]} Baxendale, P. H., Statistical equilibrium and two-point motion for a stochastic flow of diffeomorphisms, \emph{Spatial Stochastic Processes} (Progress in Probability \textbf{19}), pp189--218.
\\ \\
\textrm{[KN04]} Kleptsyn, V.~A., Nalskii, N.~B., \emph{Contraction of orbits in random dynamical systems on the circle}, Functional Analysis and Its Applications \textbf{38}(4), pp267--282.
\\ \\
\textrm{[TMHP01]} Toral, R., Mirasso, C. R., Hernandes-Garcia, E., Piro, O., Analytical and numerical studies of noise-induced synchronization of chaotic systems, \\ arXiv:nlin/0104044v1 [nlin.CD].
\\ \\
\textrm{[MT83]} Matsumoto, K., Tsuda, I., Noise-Induced Order, \emph{Journal of Statistical Physics} \textbf{31}(1), pp87-106.
\\ \\
\textrm{[KS12]} Kuksin, S., Shirikyan, A., \emph{Mathematics of Two-Dimensional Turbulence}, Cambridge Tracts in Mathematics 194, Cambridge University Press, Cambridge~New~York.
\\ \\
\textrm{[LeJan87]} Le Jan, Y., \'{E}quilibre statistique pour les produits de diff\'{e}omorphismes al\'{e}atoires ind\'{e}pendants, \emph{Annales de l'Institut Henri Poincar\'{e} Probabilit\'{e}s et Statistiques} \textbf{23}(1), pp111--120.
\\ \\
\textrm{[Kai93]} Kaijser, T., On stochastic perturbations of iterations of circle maps, \emph{Physica~D} \textbf{68}(2), pp201--231.
\\ \\
\textrm{[MS99]} Mohammed, S.-E., Schuetzow, M., The stable manifold theorem for stochastic differential equations, \emph{The Annals of Probability} \textbf{27}(2), 615--652.
\\ \\
\textrm{[CS04]} Chueshov, I., Scheutzow, M., On the structure of attractors and invariant measures for a class of monotone random systems, \emph{Dynamical Systems} \textbf{19}(2), pp127--144.
\\ \\
\textrm{[Chuesh02]} Chueshov, I., \emph{Monotone Random Systems Theory and Applications}, Lecture Notes in Mathematics 1779, Springer, Berlin Heidelberg New York.
\\ \\
\textrm{[CCK07]} Caraballo, T., Chueshov, I., Kloeden, P., Synchronization of a stochastic reaction-diffusion system on a thin two-layer domain, \emph{SIAM Journal on Mathematical Analysis} \textbf{38}(5), pp1489--1507.
\\ \\
\textrm{[LDLK10]} Liu, X., Duan, J., Liu, J. and Kloeden, P., Synchronization of Dissipative Dynamical Systems Driven by Non-Gaussian L\'{e}vy Noises, \emph{International Journal of Stochastic Analysis}, vol.~2010, Article ID 502803, 13 pages.
\\ \\
\textrm{[Car85]} Carverhill, A., Flows of stochastic dynamical systems: ergodic theory, \emph{Stochastics} \textbf{14}(4), pp273--317.
\\ \\
\textrm{[BBD14]} Bochi, J., Bonatti, C., D\'{i}az, L.~J., Robust vanishing of all Lyapunov exponents for iterated function systems, \emph{Mathematische Zeitschrift} \textbf{276}(1--2), pp469--503.
\\ \\
\textrm{[CF98]} Crauel, H., Flandoli, F., Additive Noise Destroys a Pitchfork Bifurcation, \emph{Journal of Dynamics and Differential Equations} \textbf{10}(2), pp259--274.
\\ \\
\textrm{[Crau01]} Crauel, H., Random point attractors versus random set attractors, \emph{Journal of the London Mathematical Society} \textbf{63}(2), pp413--427.
\\ \\
\textrm{[Crau02]} Crauel, H., Invariant measures for random dynamical systems on the circle, \emph{Archiv der Mathematik} \textbf{78}(2), pp145-154.
\\ \\
\textrm{[CDLR13]} Callaway, M., Doan, T.~S., Lamb, J.~S.~W., Rasmussen, M., The dichotomy spectrum for random dynamical systems and pitchfork bifurcations with additive noise, arXiv:1310.6166v1 [math.DS].
\\ \\
\textrm{[Hom13]} Homburg, A.~J., Synchronization in iterated function systems, arXiv:1303.6054v1 [math.DS].
\\ \\
\textrm{[KH95]} Katok, A., Hasselblatt, B., \emph{Introduction to the Modern Theory of Dynamical Systems}, Cambridge University Press, Cambridge.
\\ \\
\textrm{[FM10]} Ferenczi, S., Monteil, T., Infinite words with uniform frequencies, and invariant measures, \emph{Combinatorics, Automata and Number Theory} (ed. V.~Berth\'{e}, M.~Rigo), Encyclopedia of Mathematics and its Applications 135, Cambridge University Press, Cambridge.
\\ \\
\textrm{[New15]} Newman, J., Ergodic Theory for Semigroups of Markov Kernels (11th February 2015), \url{http://wwwf.imperial.ac.uk/~jmn07/Ergodic_Theory_for_Semigroups_of_Markov_Kernels.pdf}
\\ \\
\textrm{[AO03]} Ashwin, P., Ochs, G., Convergence to local random attractors, \emph{Dynamical Systems} \textbf{18}(2), pp139--158.
\\ \\
\textrm{[Dud02]} Dudley, R.~M., \emph{Real Analysis and Probability}, Cambridge Studies in Advanced Mathematics 74, Cambridge University Press, Cambridge.
\\ \\
\textrm{[Kif86]} Kifer, Y., \emph{Ergodic Theory of Random Transformations}, Birkh\"{a}user, Boston.

\end{document}